\documentclass{elsart}
\usepackage{amssymb}
\usepackage{graphicx}

\newtheorem{definicion}{Definition}[section]
\newtheorem{proposition}[definicion]{Proposition}
\newtheorem{theorem}[definicion]{Theorem}

\newtheorem{lemma}[definicion]{Lemma}
\newtheorem{remark}[definicion]{Remark}

\newtheorem{definition}[definicion]{Definition}
\newenvironment{proof}{ \textbf{Proof}:}{\hfill}

\begin{document}
\begin{frontmatter}
\title{Weak Reproductive Solutions for a
Convection-Diffusion Model Describing a Binary Alloy Solidification Processes}
\author[espana]{Blanca Climent-Ezquerra \thanksref{blanca}},
\thanks[blanca]{Partially supported by Ministerio de Ciencia, Innovaci\'on y Universidades, Grant PGC2018-098308-B-I00, with the collaboration of FEDER.}
\author[chile]{Mario Dur\'an \thanksref{mario}},
\thanks[mario]{Supported by FONDECYT-Chile
grant No 1040205 and No 7060025.}
\author[chile1]{Elva Ortega-Torres \thanksref{mario}},
\author[chile2]{Marko Rojas-Medar \thanksref{marko}}
\thanks[marko]{Supported by of MATH-AMSUD project 21-MATH-03 (CTMicrAAPDEs), CAPES-PRINT 88887.311962/2018-00 (Brazil), Project UTA-Mayor, 4753-20, Universidad de Tarapac\'a .}
\address[espana]{Dpto. EDAN, Facultad de Matem\'aticas, Universidad de Sevilla, Sevilla, Espana. E-mail: bcliment@us.es}
\address[chile]{Departamento de Ingenier\'\i a Matem\'atica, Universidad de Concepci\'on, Concepci\'on, Chile.
E-mail: maduran@udec.cl}
\address[chile1]{Departamento de Matem\'aticas, Universidad Cat\'olica del Norte, Antofagasta, Chile. E-mail: eortega@ucn.cl}
\address[chile2]{Departamento de Matem\'atica, Universidad de Tarapac\'a, Arica, Chile. 
E-mail: marko.medar@gmail.com}
\begin{abstract}
We study the existence of reproductive weak solutions for a system
of equations describing a solidification process of a binary alloy
confined into a bounded and regular domain in $\mathbb{R}^3$, having
mixed boundary conditions.
\end{abstract}
\end{frontmatter}

\section{Introduction}
The modeling and mathematical study of the behavior of metal alloys during the solidification 
process has attracted the attention of a great number of researchers and engineers in the last 
few years. The complex dynamics which unfold in this process have led to the development of acute 
theories which try to explain the advance of the solidification front on a macroscopic level, as 
well as a describe on a microscopic level the dendritic growth in the solid state zone. From a 
thermodynamic point of view we can point out the rigorous work of Hills \& al. \cite{Hi-Lo-Ro}, 
and from the point of view of a simple modeling, Rappaz \& Voller \cite{Ra-Vo}, who obtain results 
based on the properties of thermodynamic averages. From a numeric stance, some implemented software 
packages partially solve the evolving problem of the solidification of a multi component alloy. We 
can mention programs by Ahmad \cite{Ah} and Voller \& Prakash  \cite{Vo-Pr} in the area of Finite 
Elements, and Combeau \& Lesoult \cite{Co-Le} with respect to Finite Volumes.

From a theoretic point of view the question of existence, uniqueness, asymptotic or periodic behavior 
of solution for the evolving model remains basically open, since we are dealing with a degenerate
parabolic and free boundary problem of high mathematical complexity.
Some preliminaries works are Berm\'udez \& Saguez [5], Donnelly [7] and Luckhaus \& Visintin [19]. For 
the associated stationary model, about the existence of solution, uniqueness and computing aspects,
we can mention the works by Blanc \& al. \cite{Bl-Gas-Ra}, Duran \& al. \cite{Du-Or-St}, Duran \& al. 
\cite{Du-Or-Ra},and Badillo \& al. \cite{Ba-Du-Or}, Gasser \cite{Gas}, Gaillard \& Rappaz \cite{Ga-Ra}, 
among others.
It is known that certain dynamical systems may not have periodic solutions because there exist many orbits, 
or branches of bifurcations, that can be randomly reached by the solution. However, several of these systems 
are still of the {\it reproductive type}, in the sense that there exist at least two different times where 
the solution takes the same value.

In this paper, we deal with the existence of weak reproductive solutions of the system modeling  the 
solidification process.  It is organized as follows: In section 2, we state the mathematical model of 
solidification for a binary alloy. In section 3, we state some notations and known inequalities, and we define 
the notion of reproductive weak solution. In section 4, we study a regularized problem associated to the 
solidification model. In section 5, we prove the existence of a reproductive weak solutions of the regularized 
problem studied in Section 4. The results of Section 5 are employed in Section 6, where we prove the main
result about the existence of reproductive weak solution, which is obtained as the limit of a sequence of 
reproductive weak solutions of regularized problems.

\section{Mathematical model for a binary alloy solidification}
In this section we explain the mathematical model we will use to study the solidification process.
\subsection{The physical problem}
 The defective solidification of metal parts during their mould is one of the main reasons for failure before 
 the time envisioned by the industrial design. The most popular procedure for molding of metal parts consists 
 in pouring a material (pure or compound) in its liquid state at a high temperature in a mould. Once the incandescent 
 matter has been poured, it will loose heat and begin to solidify. Depending on various factors, amongst them 
 the temperature gradient and the concentration of solute in the solvent, it is possible that, during the 
 solidification inside the mould, the fluid may tend to separate into regions with different characteristics, 
 creating solute nodules surrounded by solvent with a greater level of purity than that in the original material. 
 An important consequence of this segregation is that the resulting metallic part will have, in some sectors, 
 metallurgical properties that differ with those calculated. Amongst these, the most important ones are the 
 tensile strength and malleability.
The binary alloy is poured into the mould in such a way that the lowest part is in contact with the bottom wall 
at low temperature, meanwhile the top part remains at fusion temperature. This can create a strong temperature 
gradient, due to which the process of solidification will produce an heterogeneous material. In the top of domain 
part we will see a liquid section, while at the bottom we will have already solid material, leaving a middle zone 
in which the process of solidification is not yet complete. This section of the mixture is characterized by the 
coexistence of liquid and solid material.
\subsection{Solidification model for a binary alloy}
In a binary alloy the element called solute is the one that is dissolved in the other one, called solvent. Without 
loss of generality, we can assume that in the liquid state the mixture is homogeneous. Therefore, at the beginning 
of the process a solidification front appears, which consists of solid material dendrites, around which the material 
in liquid state will adhere until we reach a completely solid phase. As it known, this process of solidification is 
fundamentally controlled by the thermodynamical variables of solute concentration and temperature in the mould. The 
variables determine the state of the matter trough the phase diagram (see Figure 2). As a fundamental reference we 
can point out, among others, the texts by Milne-Thomson \cite{Mil}, Shapiro \cite{Sha} and the doctoral theses of 
Ahmad \cite{Ah}, Gaillard \cite{Ga} and Gasser \cite{Gas}.

In the domain we can distinguish three regions, depending on the concentration $c$ and the temperature $\theta$. 
Denoting with $\Omega$ the mould  into which the melted matter has been poured, we call $\Omega_l, \Omega_m$ and 
$\Omega_s$ the regions occupied by matter in liquid, mixture (coexistence of solid and liquid) and solid state, 
respectively. Furthermore, we identify the interfaces between these three sub-domains, calling $\Gamma_{ml}$ the
mixture-liquid interface, and  $\Gamma_{sm}$ the one corresponding to solid-mixture. With the purpose of considering 
borders conditions, we will also fix some notation for the three regions of the exterior frontier $\Gamma$. This regions 
are denoted by $\Gamma_t, \Gamma_b$ and $\Gamma_v$ (see Figure 1).
\begin{center}
\includegraphics[width=7cm,height=4.5cm]{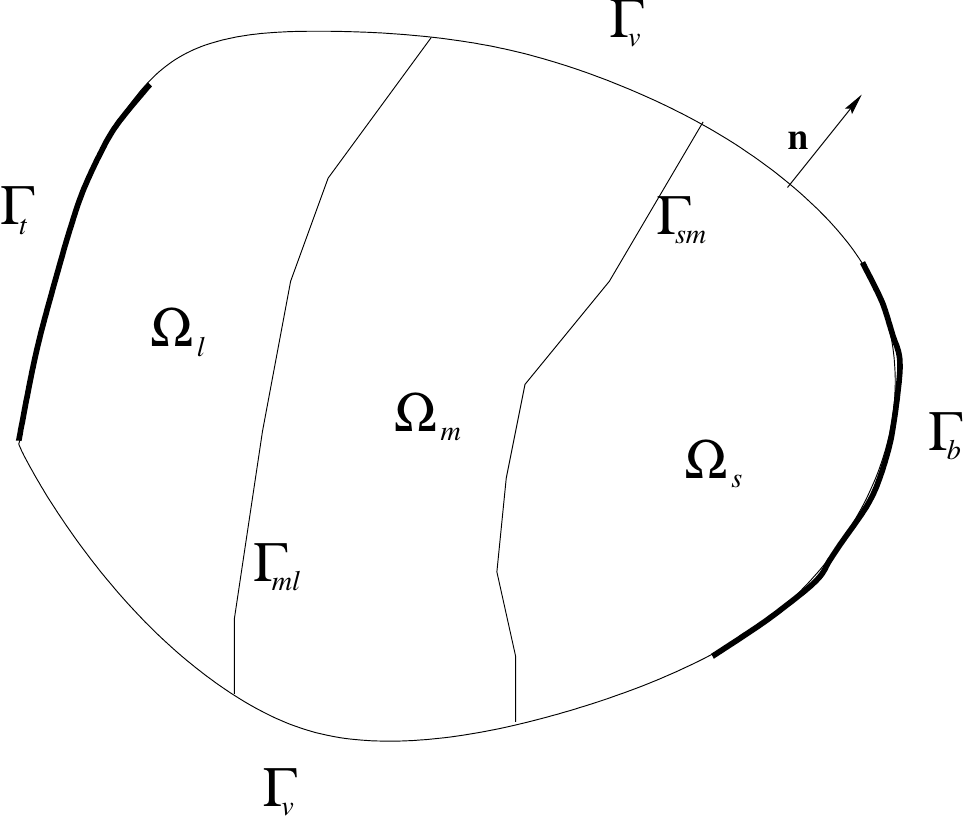}
\end{center}

\centerline{{\small{\bf Figure 1.} Solidification mould}}

In presence of thermodynamic equilibrium, we use the {\it phase diagram} to determine the state of the alloy. 
As a fundamental reference we can point out the doctoral theses by
Ahmad \cite{Ah}, Gasser \cite{Gas} and  Gaillard \cite{Ga}.
In this diagram (see Figure 2), the horizontal axis corresponds to the concentration $c$, the vertical one 
corresponds to the temperature $\theta$ and  $\theta_F$ and $\theta_E$ are the fusion and eutectic temperature 
respectively. Therein we represent two regular curves: the {\it liquids} $\gamma_\ell(\theta)$
and the {\it solidus} $\gamma_s(\theta)$, jointing the point $(0,\theta_F)$ of fusion of solvent pure with 
the points $(\gamma_l(\theta_E), \theta_E)$ and $(\gamma_s(\theta_E), \theta_E)$, respectively, where 
$\gamma_s(\theta_E)=c_A$ and $\gamma_l(\theta_E)=c_E$. They represent the temperatures of solidification 
of start and end  in the equilibrium.
\begin{center}
\includegraphics[width=11cm,height=8cm]{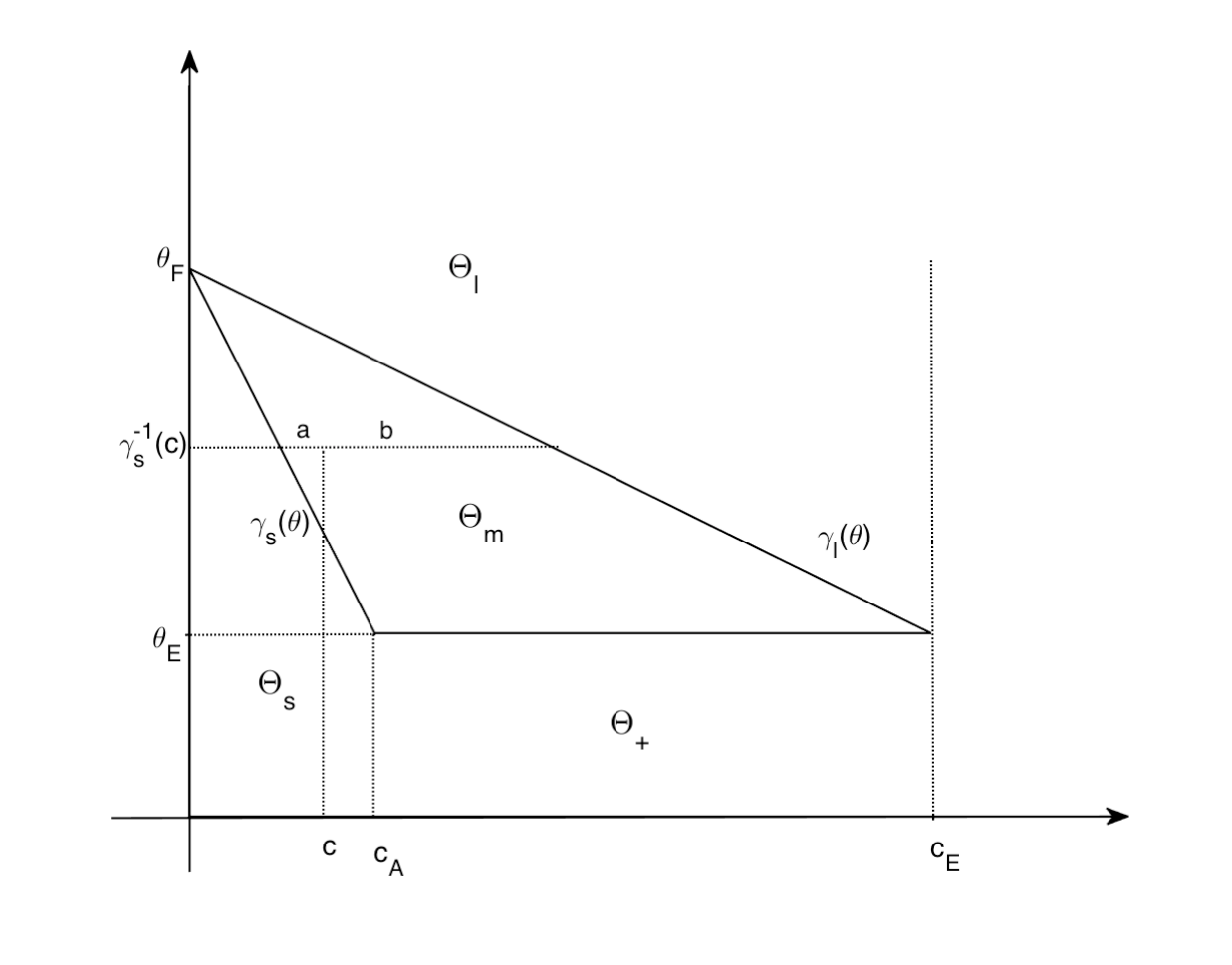}
\end{center}

\centerline{{\small{\bf Figure 2.} Phase Diagram}}

In the {\it phase diagram} we identify three zones: the liquid $\Theta_l$,
the mixture $\Theta_m$ and the solid $\Theta_s$, which are defined as:
\begin{eqnarray*}
\Theta_l&=&\{(c,\theta) \in \mathbb{R}^2 \,| \ 0<c<c_E,\:\:\theta>\gamma_l^{-1}(c)\},\\
\Theta_m&=&\{(c,\theta) \in \mathbb{R}^2 \,| \ \theta_E<\theta<\theta_F,\:\:\gamma_s(\theta)<c<\gamma_l(\theta)\},\\
\Theta_s&=&\{(c,\theta) \in \mathbb{R}^2 \,| \ 0<c<c_A,\:\:\theta<\gamma_l^{-1}(c)\},
\end{eqnarray*}
Moreover, we denote $\Theta_\pm=(\mathbb{R}_\pm\times \mathbb{R})\setminus\overline{\Theta_l\cup\Theta_s\cup\Theta_m}$.

In a stationary case, in thermodynamic equilibrium, the sub-domains $\Omega_l, \Omega_m$ and $\Omega_s$ can be described
by the thermodynamics variables of concentration and temperature,
$(c,\theta)$ as follows:
\begin{eqnarray*}
\Omega_a&=&\{x \in \Omega \,| \ (c(x), \theta(x))\in \Theta_a\},\, a=l, m, s,\\
\Gamma_{ml}&=&\{x \in \Omega \,| \ c(x)=\gamma_l(\theta(x))\}=\partial\Omega_m\cap\partial\Omega_l,\\
\Gamma_{sm}&=&\{x \in \Omega \,| \ c(x)=\gamma_s(\theta(x))\}=\partial\Omega_s\cap\partial\Omega_m,\\
\Omega_{ml}&=&\Omega_m \cup \Gamma_{ml} \cup \Omega_l.
\end{eqnarray*}
In order to model the transport phenomenon in the mixture region, we denote by $c_l$ and $c_s$ the concentration 
of solute in liquid and solid phase, respectively, which can be precisely determined through
the {\it phase diagram} (see e.g. \cite{Du-Or-St},  \cite{Er},  \cite{Ga}), then, $c_l=c$ in $\Theta_l$ and 
$c_l=\gamma_\ell (\theta)$ if $ (c,\theta) \in \overline{\Theta_m}\setminus(\partial \Theta_{-}\cup \partial \Theta_{+})$.
Artificially, we extend the definition of $c_l$ to whole $\mathbb{R}^2$ as:
\begin{equation}\label{eq1}
c_l(c,\theta)=\left\{
\begin{array}{ll}
c & \ \mbox{ if } \ (c,\theta) \in \Theta_l,\\
\gamma_\ell (\theta) & \ \mbox{ if } \ (c,\theta) \in \overline{\Theta_m}\setminus(\partial \Theta_{-}\cup 
\partial \Theta_{+}),\\
\gamma_\ell (\gamma_s^{-1}(c)) & \ \mbox{ if } \ (c,\theta) \in  \Theta_s,\\
0 & \ \mbox{ if } \ (c,\theta) \in \overline{\Theta_{-}},\\
c_E& \ \mbox{ if } \ (c,\theta) \in \overline{\Theta_{+}}.
\end{array}
\right.
\end{equation}
In the solid-liquid mixture region it makes sense to define a function that indicate the fraction of matter 
in solid state $f_s(c,\theta)$ and in liquid state 
$$f_l(c,\theta)=1-f_s(c,\theta),$$
per unit of volume. Thus, $f_s$ can be calculated by (see Figure 2):
$$f_s(c,\theta)=\left\{
\begin{array}{ll}
0 & \ \mbox{ if } \ (c,\theta) \in \Theta_l,\\
\displaystyle\frac{a}{a+b} & \ \mbox{ if } \ (c,\theta) \in \Theta_m,\\
\,1& \ \mbox{ if }\,(c,\theta) \in \Theta_s,
\end{array}
\right.
$$
where $a=\gamma_l(\theta)-c$ and $b=c-\gamma_s(\theta)$. Analogously, we extend this definition of $f_s$ 
to $\mathbb{R}^2$ by 1 if $(c,\theta) \in \overline{\Theta_s\cup \Theta_{+}}$ and by 0 in other case.

Moreover, the unknown function of concentration $c$ that intervenes in the equations is a weighted average 
between $c_l$ and $c_s$, that is, 
$$ c= c_l\,(1-f_s)+c_s\,f_s.$$

Taking these expressions into account, we propose to model the dynamics of the mixture from the moment it is 
poured into the mould in liquid phase until it has finished its solidification. It is clear that, as time passes, 
the zones $\Omega_a,\,a\in\{l,m,s\}$, will appear, evolve, and disappear within $\Omega$. This way, we will see 
the appearance of a concentration and a temperature gradient which, added to the action of the force of gravity 
which acts over the alloy, can develop a movement of {\it velocity} $\textit{\textbf v}(x,t)$ and {\it pressure field} 
$p(x,t)$ in the interior of the mixture $\Omega_{ml}\times (0,T)$ (see \cite{Ga-Ra}). It is important to observe that 
the set $\Omega_{ml}$ and the interfaces $\Gamma_{sm}, \Gamma_{ml}$ are unknown a priori, because its dependence on 
the concentration $c$ and temperature $\theta$. This is an additional difficult of the problem.

We model the transport during the solidification process using the incompressible Navier-Stokes equations, coupled 
with terms of force (external and internal) as follows (see \cite{Ga})
\begin{eqnarray}
\hspace*{-.7cm}\rho\frac{\partial\textit{\textbf v}}{\partial t} -2\nu \, {\rm div}\,e(\textit{\textbf v})
+\rho(\textit{\textbf v} \cdot \nabla)\, \textit{\textbf v} + F_i(c, \theta)\, \textit{\textbf v}
+\nabla  p &=&\mathbf{F}_e(c, \theta) \ \ {\rm in} \ \mathcal{Q}_{ml},\label{eq2}\\
{\rm div} \, \textit{\textbf v}& =&0 \ \ {\rm in} \ \mathcal{Q}_{ml},\label{eq3}
\end{eqnarray}
where $\mathcal{Q}_{ml}=\Omega_{ml}\times (0,T)$, $\nu$ is the {\it dynamical viscosity}, $\rho$ is 
the {\it mean density} in the mixture, both considered constants, and $e(\textit{\textbf v})
=( \nabla \textit{\textbf v}+(\nabla \textit{\textbf v})^t)/2$ is the {\it linearized stress tensor} 
of $\textit{\textbf v}$. The external force ${\mathbf F}_e (\cdot,\cdot)$ 
is given by the Boussinesq approximation:
\begin{equation}\label{eq4}
\mathbf{F}_e(c, \theta)=\rho\, \mathbf{g}\,(\alpha\,(\theta-\theta_r)
+\beta \, (c_l(c, \theta)-c_r)),
\end{equation}
where $\alpha, \beta$ are known real constants, $\theta_r, c_r$ are respectively a temperature and a concentration
of reference, and $\mathbf{g}= (0, 0, - g)^t$ is the gravitational force.

The term $F_i(\cdot,\cdot)$ represents the {\it internal force} within a porous material which opposes the movement 
of the fluid on $\Omega_{ml}$. It is known as the {\it Carman-Kozeny law}, and it is given by
\begin{equation}\label{eq5}
F_i(c, \theta)= C_0\, \frac{f_s^2(c, \theta)}{(1-f_s(c,
\theta))^3} \quad \mbox{ with } C_0>0.
\end{equation}
In the liquid region $F_i(\cdot, \cdot) = 0$ and is hypersingular in $\Omega_{ml}$. On the other hand, we suppose 
that the velocity field $\textit{\textbf v}$ is $0$ on $\Omega_s$ and continuous over the whole domain $\Omega$.

The diffusion process of the concentration and the heat transfer in a multi-component mixture is described by a 
non-linear system of parabolic equations. The concentration, the temperature and the velocity are the variables 
involve. These equations take into account the phenomena of diffusion and transport, and are deduced considering
the laws of solute conservation in the mixture and conservation of energy. Simplifying these equations, using the 
laws of Fourier and Fick, we obtain
\begin{eqnarray}
\frac{\partial c}{\partial t}- \eta \,\Delta \, c+\textit{\textbf v} \cdot\nabla  c_l (c, \theta)&=&0\quad 
\mbox{ in } \ \mathcal{Q},\label{eq6} \\
\frac{\partial}{\partial t} \,H(c, \theta)-\kappa\, \Delta\,\theta+ \rho \, C_p \, \textit{\textbf v} \cdot 
\nabla\theta & =&0 \quad \mbox{ in } \mathcal{Q},\label{eq7}
\end{eqnarray}
where $\mathcal{Q}=\Omega\times (0,T)$, $\kappa, \eta$ are two strictly positive real constants that represent 
the physical diffusions, $C_p$ is the \textsl{caloric capacity}, $c_l(\cdot,\cdot)$ is the liquid concentration 
and $H(\cdot,\cdot)$ is the \textsl{enthalpy}, defined as
$$H(\cdot,\cdot)= \rho\, C_p\,\theta+ \rho\,L_f\, f_l(c,\theta),$$
where $L_f$ is the latent heat of fusion. There follows, we approximate $H(\cdot,\cdot)$ by  $ \rho \,C_p \,\theta$ 
because we do not know how to treat the non-linear term $\partial_t \frac{\gamma_l(\theta)-c}{\gamma_l(\theta)
-\gamma_s(\theta)}$ that appears in (\ref{eq7}) in  $\Omega_m$.

The initial conditions are obtained from the moment in which the melted material is poured into the solidification 
mould and are, in general, regular functions. The limit conditions considered are:
\begin{itemize}
\item Adherence of the fluid over $\Gamma_{bt}=\Gamma_b\cup\Gamma_t$.
\item The normal flux and the tangential stress are zero over $\Gamma_v$.
\item The normal concentration flux is zero over $\Gamma$ and a fixed total concentration of solute in the domain. 
The temperature is known over $\Gamma_{bt}$ and the wall $\Gamma_v$ is adiabatic.
\end{itemize}

\subsection{Setting the evolutive mathematical model} 
Let be $\Omega$ an open domain subset of $\mathbb{R}^3$ with regular boundary $\Gamma =\partial\Omega$. Then, 
from (\ref{eq2})-(\ref{eq7}) and the boundary conditions stated above, the problem of our present study reads: 
Find functions $(c, \theta, \textit{\textbf v}, p): \mathcal{Q} \rightarrow \mathbb{R}^6$, satisfying
\begin{eqnarray}
&& \frac{\partial c}{\partial t}-\eta\, \Delta c +\textit{\textbf v}\cdot\nabla c_l (c, \theta) = 0 
\quad {\rm in} \ \mathcal{Q},\label{eq8}\\
&& \frac{\partial c}{\partial n} = 0 \quad {\rm on} \ \Gamma\times (0,T), \qquad \int\limits_\Omega 
c \,(x,t) \, dx = c_g \quad {\rm in} \ (0,T),\label{eq8a}\\
&& \ c(x,0)=c_0(x) \quad {\rm in} \ \Omega, \label{eq8b}
\end{eqnarray}
\begin{eqnarray}
&&\frac{\partial \theta}{\partial t}-\kappa \Delta \theta +\rho\,
C_p\,\textit{\textbf v}\cdot \nabla \theta = 0 \quad {\rm in} \
\mathcal{Q},\label{eq9}\\
&&\theta = \theta_\delta \quad {\rm on} \ \Gamma_{bt} \times (0,T),
\qquad \frac{\partial \theta}{\partial \textit{\textbf n}} = 0 \quad {\rm on}
\ \Gamma_v \times (0,T),\label{eq9a}\\
&&\theta(x,0)=\theta_0(x) \quad {\rm in} \ \Omega,\label{eq9b}
\end{eqnarray}
\begin{eqnarray}
&&\frac{\partial \textit{\textbf v}}{\partial t}-2 \nu\, {\rm div}\,e(\textit{\textbf v})
+\rho\, (\textit{\textbf v}\cdot \nabla)\textit{\textbf v}+F_i(c,\theta)\,\textit{\textbf v}
+\nabla p=\mathbf{F}_e (c, \theta) \ {\rm in} \ \mathcal{Q}_{ml}, \label{eq10}\\
&&{\rm div}\, \textit{\textbf v}=0 \quad {\rm in} \ \mathcal{Q}_{ml}, \ \hspace{2.8cm} 
\textit{\textbf v} =\mathbf{0} \quad {\rm in} \ \Omega_s \times (0,T),\label{eq10a}\\
&& \textit{\textbf v} = \mathbf{0} \quad {\rm on} \ \Gamma_{bt} \times (0,T), \ \hspace{2.1cm}
\textit{\textbf v}\cdot \textit{\textbf n} = 0 \ {\rm on} \quad \Gamma_ v \times (0,T),\label{eq10b}\\
&&\sigma (\textit{\textbf v}, p) \, \textit{\textbf n} \wedge \textit{\textbf n} = 0
\quad {\rm on} \ \Gamma_v\times (0,T), \ \ \quad \textit{\textbf v}(x,0)
= \textit{\textbf v}_0(x) \quad {\rm in} \ \Omega.\label{eq10c}
\end{eqnarray}
Here $c_g$ is the total quantity of solute in the mixture and is a strictly positive real constant that satisfies 
the compatibility relation
\begin{equation}\label{eq11}
0\leq c_g \leq \gamma_\ell(\theta_E)|\Omega|.
\end{equation}
Furthermore, $\theta_\delta$ is a known distribution of temperature on $\Gamma_{bt}$, $n$ denote the vector 
unit outward to the surface $\Gamma$ and $\sigma(\textit{\textbf v},p)$ represents the stress tensor given by 
$\sigma(\textit{\textbf v}, p)= -p\,I+2\,\nu\,e(\textit{\textbf v})$.
\section{Function spaces and preliminaries}
The vector-valued functions in $\mathbb{R}^3$ are denoted by $\{\,\textit{\textbf{u}}, \,\textit{\textbf{v}},
\,\textit{\textbf{w}}\,\}$, while the scalar ones are written solely by $\{\,p,\,\varphi,\,\psi,\,\phi,\,c,\,\theta \}$.
For simplicity, we denote $\mathbf{L}^2$ and $\mathbf{H}^1$ instead of $L^2(\Omega)^3$ and $H^1(\Omega)^3$, $L^p(X)$ 
instead of $L^p(0,T; X)$, for $1\leq p\leq \infty$, etc. Also, the scalar product in $L^2$ will be denoted by 
$(\cdot,\cdot)$, and $\langle\cdot,\cdot\rangle$ will denote some duality products. For all integer $m\geq 0$, 
we denote the norm of the scalar and vector-valued Hilbert spaces $H^m$ and $L^2$ as $\Vert\cdot\Vert_m$ and 
$\Vert\cdot\Vert$ respectively. The $L^p$ norma will be denoted by $\Vert\cdot\Vert_{L^p}$. Other notation that 
we use are: $H^1(H^1)$ to denote the space $H^1(0,T; H^1(\Omega))$ and $W^{1, \infty}(\Omega)$ denote the usual 
Sobolev space  which consists of functions in $L^\infty$ such that their derivatives in 
the sense of the distributions are in $L^\infty$.

Let us consider the following function spaces:
\begin{eqnarray*}
\textit{\textbf H}&=&\{ \textit{\textbf u} \in\mathbf{L}^2: {\rm div}\,\textit{\textbf u}=0 \mbox{ in } \Omega,
\textit{\textbf u}\cdot \textit{\textbf n} =0 \mbox{ on } \partial \Omega \},\\
\textit{\textbf V}&=&\{ \textit{\textbf u} \in \mathbf{H}^1_0: {\rm div}\,\textit{\textbf u}=0 \mbox{ in } \Omega\},\\
\textit{H}_\theta & =& \{ \, \psi \in \textit{H}^1: \psi= 0 \ {\rm on} \ \Gamma_{bt}\},\\
\textit{\textbf H}_v& = &\{ \, \textit{\textbf w} \in \mathbf{H}^1: {\rm div}\, \textit{\textbf w} = 0 \mbox{ in } 
\Omega,\ \textit{\textbf w} = 0 \mbox{ on } \Gamma_{bt} \mbox{ and } \textit{\textbf w} \cdot \textit{\textbf n} 
= 0 \mbox{ on } \Gamma_{v}\}.
\end{eqnarray*}
From now on, $C>0$ will denote different constants, depending only on the fixed data of the problem.

In forthcoming section of the paper we use some inequalities, whose proofs can be found in \cite{La}, and that we
state in what follows:
\begin{lemma}\label{Lemma 3.1} (Young's Inequality) 
Let $A\geq 0, \, B\geq 0$ be real constants. Suppose that $p, \, q \in \mathbb{R}$ such that $ \, 1\leq p < \infty\,$ 
and $\, \displaystyle \frac{1}{p} + \frac{1}{q} = 1$. Then, for all $ \, \xi > 0\, $, the following inequality hold:
\begin{equation}\label{eq12}
A\, B \leq \frac{\xi}{p}\, A^p + \frac{(p - 1) \,\xi^{1-p}}{p} \, B^q.
\end{equation}
\end{lemma}
\begin{lemma}\label{Lemma 3.2} (Gronwall's Inequality) 
Let functions $g(t)$ and $h(t)$ nonnegative, continuous with $ \, g^{'}(t) \geq 0\, $ and $h(t)$ sumable on $[0,T]$. 
Assume that it is satisfied
\[ \varphi\,(t) + \int_0^t \varphi^*\,(\tau)\, d\,\tau \leq\frac{g(t)} {\lambda} + \int_0^t h\, (\tau)\, 
\varphi\, (\tau)\, d\, \tau, \]
for two positive continuous functions $ \, \varphi \,$ and $\, \varphi^*\,$ on $\, [0,T]\,$ and the constant 
$\, \lambda > 0$. Then it holds
\begin{equation}\label{eq13}
\varphi\,(t) + \int_0^t \varphi^*\,(\tau)\, d\,\tau \leq \frac{g(t)} {\lambda}\, (1 + \int_0^t h\, (\tau)\, 
d\, \tau \,) \,\exp\, ( \int_0^t h\, (\tau)\, d\, \tau\,).
\end{equation}
\end{lemma}
\begin{lemma}\label{Lemma 3.3} (\textit{Friedrich's Lemma}) Let
$\Omega$ be a fixed bounded domain of $\mathbb{R}^3$ with boundary $\, \partial \Omega\,$ of class $\, C^2$. 
For any  $\omega > 0$, $N_\omega$ basis functions ${\textit{\textbf w}}^i$,  $i=1,2,\dots, N_\omega$, can be constructed satisfying 
the following inequality
\begin{equation}\label{eq14}
\|\textit{\textbf u}\|^2 \leq \sum_{i=1}^{N_\omega}\,(\textit{\textbf u},{\textit{\textbf w}}^i)^2 +\omega \, 
\|\nabla\textit{\textbf u}\|^2 \quad \forall \,\textit{\textbf u} \in\mathbf{H}^1_0.
\end{equation}
\end{lemma}
We denote the classical bilinear and trilinear forms by
\[(\nabla \textit{\textbf u}, \nabla  \textit{\textbf v})=\sum_{i,j=1}^3 \int_\Omega \frac{\partial u_j}{\partial x_i}
\frac{\partial v_j}{\partial x_i}\, dx, \quad (\textit{\textbf u}\cdot \nabla  \textit{\textbf v}, \textit{\textbf w})
= \sum_{i,j=1}^ 3\int_\Omega u_j \,\frac{\partial v_i}{\partial x_j}\, w_i \, dx, \]
for all vector-valued functions $ \, \textit{\textbf u},\textit{\textbf v}, \textit{\textbf w}$ for which the integrals
make sense. Moreover, the trilinear form has the properties
\[(\textit{\textbf u} \cdot \nabla  \textit{\textbf v}, \textit{\textbf w}) =-(\textit{\textbf u} \cdot \nabla
\textit{\textbf w}, \textit{\textbf v}), \quad (\textit{\textbf u}\cdot \nabla  \textit{\textbf v}, \textit{\textbf v}) 
= 0\quad \forall \, \textit{\textbf u} \in \textit{\textbf H}_v,\ \forall \, \textit{\textbf v}, \,
\textit{\textbf w} \, \in \mathbf{H}^1.\]
Hereinafter, we will use the following classical interpolation and Sobolev
inequalities (for $3D$ domains):
$$\Vert v\Vert_{L^6}\leq C\Vert v\Vert_1, \quad
\Vert v\Vert_{L^3}\leq\Vert v\Vert^{1/2}\Vert
v\Vert_1^{1/2}\quad\forall v \in H^1$$ and
$$\Vert v\Vert_{L^\infty}\leq C \Vert v\Vert_1^{1/2}\Vert v\Vert_2^{1/2}\quad\forall v \in H^2.$$
\vspace{2mm}

We state now our problem rigorously establishing regularity assumptions on the boundary $\Gamma $ and 
some technical assumption on the external data.
\vspace{2mm}

\noindent {\bf Hypothesis:}
\begin{itemize}
\item[({\bf S$_1$})]
{\it Let $O_0$ be a neighborhood of the
origin. Thus $O_0\subseteq{\rm int\ }\Omega$, $\Omega$ is
connected and} $\,\Omega\subseteq B(0,R),\,R>0$.
{\it The boundary} $\Gamma \in C^2$;
\item[({\bf S$_2$})]
{\it The given external field of
temperature $\theta_\delta$ belongs to
$H^1(L^\infty(\Gamma_{bt}))$ and there exists
a continuous extension in $H^1(H^1)\cap L^2(H^1)$, satisfying
$\theta_\delta(x,0)=\theta_\delta(x,T)$ in $\Omega$;}
\item[({\bf S$_3$})]
{\it The function of concentration in
liquid phase $c_l(\cdot,\cdot)$ belongs to
$W^{1,\infty}(\mathbb{R}^2)$}.
\end{itemize}
\vspace{2mm}

There follows, we define precisely the notion of a weak reproductive solution for the system (\ref{eq8})-(\ref{eq10c}).
\begin{definition}\label{Definition 3.4} 
We say that the triplet of functions $(c(x,t),\,\theta(x,t),\,\textit{\textbf v}(x,t))$ is a weak reproductive 
solution of the evolutive problem (\ref{eq8})-(\ref{eq10c}) at time $T$ if and only if there exists
$c_0,\theta_0\in\,L^2$ and $\textit{\textbf v}_0 \in\,\textit{\textbf H}$ such that the following relations hold:
\begin{eqnarray}
&& c(x,T)=c_0(x),\quad \theta(x,T)=\theta_0(x),\quad
\textit{\textbf v}(x,T)=\textit{\textbf v}_0(x)\mbox{ a.e. in }\Omega,\\
&&c \in L^2(H^1)\cap L^\infty(L^2),\quad\ c(x,0)=c(x,T)\mbox{ in }\Omega,\\
&&\theta \in L^2(H^1)\cap L^\infty(L^2),\quad\ \theta(x,0)=\theta(x,T)\mbox{ in }\Omega,\\
&&\textit{\textbf v} \in L^2(\textit{\textbf H}_v)\cap L^\infty(\textit{\textbf H}),\quad \
\textit{\textbf v}(x,0)=\textit{\textbf v}(x,T)\mbox{ in }\Omega,\\
\noalign{\medskip}
&&0 \leq c \leq \gamma_{\ell}\,(\theta_E) \ \mbox{ in } \mathcal{Q},\label{eq18}\\
&&\int\limits_\Omega c\,(x,t) \, dx = c_g \ \mbox{ in } (0,T),\label{eq19}\\
&&\theta (\cdot,t)- \theta_\delta(\cdot,t) \in H_\theta \ \mbox{ in } (0,T),\label{eq20}\\
&&\textit{\textbf v}=0 \quad \mbox{ in }\Omega_s \times (0,T),\label{eq21}
\end{eqnarray}
and the equations are satisfies in a distributional sense:
\begin{eqnarray}
&&\hspace*{-1cm}(\frac{\partial c}{\partial t}, \varphi)+\eta \,(\nabla c,\nabla \varphi)+(\textit{\textbf v}\cdot
\nabla c_l(c, \theta), \varphi)= 0,\label{eq15}\\
&&\hspace*{-1cm}(\frac{\partial \theta}{\partial t}, \psi)+\kappa\,(\nabla \theta,\nabla \psi)
+\rho\,C_p\,(\textit{\textbf v}\cdot \nabla \theta, \psi )= 0,\label{eq16}\\
&&\hspace*{-1cm}(\frac{\partial \textit{\textbf v}} {\partial t},\, \textit{\textbf w})
+\nu \,(\nabla \textit{\textbf v},\,\nabla \textit{\textbf w})+\rho\, ((\textit{\textbf v} \cdot
\nabla )\textit{\textbf v},\,\textit{\textbf w})+(F_i(c,\theta) \, \textit{\textbf v},\, 
\textit{\textbf w})=(\mathbf{F}_e (c, \theta),\, \textit{\textbf w}),\label{eq17}
\end{eqnarray}
for all $\varphi\in H^1,\psi\in H_\theta, \textit{\textbf w}\in \textit{\textbf H}_v$ with 
${\rm supp}\,\textit{\textbf w} \subset {\Omega_{ml}}$.
\end{definition}
It is important to note that any reproductive weak solution of the problem (\ref{eq8})-(\ref{eq10c}) 
satisfy the following result.
\begin{proposition}\label{Proposition 3.5} 
Let $ (c, \theta, \textit{\textbf{v}})$ be a weak reproductive solution of the problem
$(\ref{eq8})-(\ref{eq10c})$. Then the following inequalities are satisfied
\begin{eqnarray}
\hspace*{2.5cm}0 \leq & c(x,t) & \leq \gamma_\ell(\theta_E) \ \mbox{ a.e. in } \mathcal{Q},\label{eq23}\\
\theta_\delta^{\inf} \leq \, &\theta(x,t)& \, \leq \,\theta_\delta^{\sup} \ \mbox{ a.e. in } \mathcal{Q},\label{eq24}
\end{eqnarray}
where $ \theta_\delta^{\inf}=\min\{\mathop{\inf}\limits_{\Omega}
\theta_0(x),\,  \mathop{\inf}\limits_{\Lambda} \theta_\delta (x,t)\},
\ \theta_\delta^{\sup}=\max\{\mathop{\sup}\limits_{\Omega}
\theta_0(x),\, \mathop{\sup}\limits_{\Lambda} \theta_\delta (x,t)\}$
and $\Lambda= \Gamma_{bt}\times (0, T)$.
\end{proposition}
\begin{proof}
By definition $f_+=\max\{0, f\}$ and $f_-=\max\{0, -f\}$.
To prove the inequality on the right side of (\ref{eq23}), we take
$\varphi = [c (\cdot,t) - \gamma_\ell(\theta_E)]_+ \in
H^1$ as test function in (\ref{eq15}), by a simple calculation, we obtain
\begin{eqnarray}\label{eq25}
&&\frac{1}{2} \frac{d}{dt} \|[c (t) -\gamma_\ell (\theta_E)]_+\|^2
+ \eta\,\|\nabla  [c (t)-\gamma_\ell( \theta_E)]_+\|^2\nonumber \\
&&\hspace*{1cm}+(\textit{\textbf v} \cdot \nabla c_l(c,\theta),\,[c\,(t) - \gamma_\ell(\theta_E)]_+ )=0.
\end{eqnarray}
From (\ref{eq1}), if $c (t) \geq\gamma_\ell (\theta_E)$, then $(c,\theta)\in \Theta_+$ and $c_l$ is constant, therefore
$$(\textit{\textbf v}\cdot \nabla c_l(c, \theta ),\,[c(t)-\gamma_\ell (\theta_E)]_+) = 0.$$
By integrating (\ref{eq25}) with respect to $t$ from $0$ to $T$  and since $c(T)=c(0)$, we get
\begin{equation}\label{eq27}
\eta \int_0^T \|\nabla [c (t)-\gamma_\ell(\theta_E)]_+\|^2 dt = 0,
\end{equation}
and we conclude that $ [c (x,t) - \gamma_\ell(\theta_E)]_+ = M$, being $M$ a non-negative constant. 
By definition, $M= 0$ or $M=c (x,t) - \gamma_\ell(\theta_E)$, then if $M=c (x,t) - \gamma_\ell(\theta_E)$ 
we have $c (x,t)=M +\gamma_\ell(\theta_E)$, so from (\ref{eq19}) and by integrating, we deduce that 
$c_g - \gamma_\ell(\theta_E)|\Omega| = M\, |\Omega| > 0$, thus $c_g > \gamma_\ell(\theta_E)|\Omega|$ in 
contradiction with (\ref{eq11}). Therefore, $[c (t)- \gamma_\ell(\theta_E)]_+=0$
which implies that $c (t)- \gamma_\ell (\theta_E)\leq 0$, obtaining the upper bound of (\ref{eq23}).

To prove that $c$ is not negative we use the test function $ \varphi = [c (\cdot,t)]_- \in H^1$, and 
from (\ref{eq15}) similarly as to (\ref{eq27}), we can obtain
\[\eta \int_0^T \|\nabla [c (t)]_-\|^2 dt = 0,\]
which implies $[c (x,t)]_- = M_1$ being $M_1$ a non-negative constant.
By definition $M_1 = 0$  or $M_1 = - c(x,t)$, then if $M_1 = -c(x,t)$ from  (\ref{eq19}), we get $ c_g = - M_1 \, 
|\Omega| < 0$ in contradiction with (\ref{eq11}). Thus, $[c (x,t)]_- = 0$ and the lower bound of (\ref{eq23}) is proved.

To prove the left side inequality of  (\ref{eq24}), we choose $\psi =[\theta (\cdot,t)-\theta_\delta^{sup} ]_+ 
\in H_\theta$ as function test in (\ref{eq16}) and taking into account that
$$(\textit{\textbf v}\cdot \nabla \theta, [\theta-\theta_\delta^{sup} ]_+ )=(\textit{\textbf v}
\cdot \nabla  [\theta-\theta_\delta^{sup} ]_+, [\theta-\theta_\delta^{sup} ]_+ )=0,$$
we obtain
\begin{equation}\label{eq28}
\frac{d}{dt}\|[\theta\, (t)-\theta_\delta^{sup}]_+ \|^2 
+ 2\,\kappa \|\nabla [\theta\, (t)-\theta_\delta^{sup}]_+ \|^2=0.
\end{equation}
Since $ \theta (T)= \theta (0)$ a.e. in $\Omega$, by integrating (\ref{eq28}) with respect to $t$ from $0$ to $T$, we get
\[\int_0^T \|\nabla  [\theta (t)-\theta_\delta^{sup}]_+\|^2 dt = 0.\]
Hence, taking into account (\ref{eq20}), we can conclude that $[\theta(x,t)-\theta_\delta^{sup}]_+ = 0$ and 
the upper bound of (\ref{eq24}) is obtained. The lower bound of (\ref{eq24}) is proved similarly, using the
test function $\psi=[\theta (\cdot,t)-\theta_\delta^{inf}]_ -\in H_\theta$.
\end{proof}

\section{The regularized mathematical problem}
Observe that the Carman-Kozeny term given in (\ref{eq5}), becomes discontinuous when $x \in \Omega_s$ 
since $f_s(c(x,t), \theta(x,t))= 1$. Then, to avoid this singularity and to consider the system defined 
in whole the domain $\Omega$, we define a family of regularized problems, which are obtained
by modifying the Carman-Kozeny term, as follows:
\begin{equation}\label{eq29}
F_i^\epsilon \,(c_\epsilon, \theta_\epsilon)=C_0\frac{f_s^2(c_\epsilon, \theta_\epsilon)}{(1-f_s(c_\epsilon,
\theta_\epsilon)+\epsilon)^3}\quad \forall\,\epsilon \in (0,1]\,.
\end{equation}
More precisely, the family of the regularized associated problems is defined as: Find the functions 
$ (\, c_\epsilon, \,\theta_\epsilon, \, p_\epsilon, \, \textit{\textbf{v}}_\epsilon\,) : \, \mathcal{Q} 
\rightarrow \mathbb{R}^6$, such that
\begin{eqnarray}
&&\frac{\partial c_\epsilon}{\partial t}-\eta \, \Delta c_\epsilon+\textit{\textbf v}_\epsilon 
\cdot \nabla c_l (c_\epsilon,\theta_\epsilon)=0 \quad {\rm in} \ \mathcal{Q},\label{eq30}\\
&& \ \frac{\partial c_\epsilon}{\partial \textit{\textbf n}}=0 \quad {\rm on} \ \Gamma \times (0,T), 
\qquad \int_\Omega c_\epsilon (x,t) \, dx= c_g \quad {\rm in} \ (0,T),\label{eq30a} \\
&& \ c_\epsilon(x,0)= c_0(x) \quad {\rm in} \ \Omega,\label{eq30b}
\end{eqnarray}
\begin{eqnarray}
&&\frac{\partial \theta_\epsilon}{\partial t}-\kappa \Delta \theta_\epsilon+ \rho\, C_p\,\textit{\textbf v}_\epsilon 
\cdot \nabla \theta_\epsilon = 0 \quad {\rm in} \ \mathcal{Q}, \label{eq31}\\
&& \ \theta_\epsilon=\theta_\delta \quad {\rm on} \ \Gamma_{bt}\times (0,T),\qquad \frac{\partial \theta_\epsilon}
{\partial \textit{\textbf n}}= 0 \quad {\rm on} \ \Gamma_v \times (0,T), \label{eq31a} \\
&& \ \theta_\epsilon(x,0)=\theta_0(x) \quad {\rm in} \ \Omega,\label{eq31b}
\end{eqnarray}
\begin{eqnarray}
&&\hspace*{-.6cm} \frac{\partial \textit{\textbf v}_\epsilon}{\partial t}-2 \nu\, {\rm div}\, 
e(\textit{\textbf v}_\epsilon) + \rho\,(\textit{\textbf v}_\epsilon \cdot \nabla )
\textit{\textbf v}_\epsilon+F_i^\epsilon (c_\epsilon, \theta_\epsilon) \, \textit{\textbf v}_\epsilon
+ \nabla p_\epsilon = \mathbf{F}_e (c_\epsilon, \theta_\epsilon) \ {\rm in} \ \mathcal{Q},\label{eq32}\\
&&\hspace*{-.6cm} \ {\rm div}\, \textit{\textbf v}_\epsilon = 0 \quad {\rm in} \ \mathcal{Q}, \qquad \qquad 
\qquad \textit{\textbf v}_\epsilon= \mathbf{0} \quad {\rm on} \ \Gamma_{bt} \times (0,T),\label{eq32a}\\
&&\hspace*{-.6cm} \ \textit{\textbf v}_\epsilon \cdot \textit{\textbf n} = 0 \quad {\rm on}
\ \Gamma_v \times (0,T), \qquad \sigma (\textit{\textbf v}_\epsilon, p_\epsilon)\, \textit{\textbf \textit{\textbf n}}
\wedge \textit{\textbf n}=0 \quad {\rm on} \ \Gamma_v\times (0,T),\label{eq32b}\\
&&\hspace*{-.6cm} \ \textit{\textbf v}_\epsilon(x,0)= \textit{\textbf v}_0(x) \quad {\rm in} \ \Omega.\label{eq32c}
\end{eqnarray}

In the same way as Definition \ref{Definition 3.4}, we give the following definition of reproductive regularized weak 
solution of the problem (\ref{eq8})-(\ref{eq10c}).
\begin{definition} \label{Definition 4.1} 
We say that the triplet of functions $(c_\epsilon(x,t),$ $\,\theta_\epsilon(x,t),\, \textit{\textbf v}_\epsilon(x,t))$ 
is a regularized reproductive weak solution of the evolutive problem (\ref{eq8})-(\ref{eq10c}) at time $T$ if and 
only if there exist $c_0,\theta_0\in\,L^2$ and $\textit{\textbf v}_0\in\,\textit{\textbf H}$ such that
the following relations hold:
\begin{eqnarray}
&& c_\epsilon(x,T)=c_0(x),\,\theta_\epsilon(x,T)=\theta_0(x),\, \textit{\textbf v}_\epsilon(x,T)
=\textit{\textbf v}_0(x) \mbox{  a.e. in } \Omega,\\
&& c_\epsilon \in L^2(H^1)\cap L^\infty(L^2), \quad c_\epsilon(x,0)=c_\epsilon(x,T) \mbox{  in } \Omega,\\
&&\theta_\epsilon \in L^2(H^1)\cap L^\infty(L^2), \quad\theta_\epsilon(x,0)=\theta_\epsilon(x,T) \mbox{  in }\Omega,\\
&&\textit{\textbf v}_\epsilon \in L^2(\textit{\textbf H}_v)\cap L^\infty(\textit{\textbf H}), 
\quad\textit{\textbf v}_\epsilon(x,0)=\textit{\textbf v}_\epsilon(x,T) \mbox{  in }\Omega,\\
&&\int_\Omega c_\epsilon (x,t)\,dx = c_g \ \mbox{ in } (0,T),\quad
\theta_\epsilon (\cdot,t)- \theta_\delta(\cdot,t) \in H_\theta \ \mbox{ in } (0,T),
\end{eqnarray}
and $c_\epsilon,\,\theta_\epsilon$ and $\textit{\textbf v}_\epsilon$
satisfy the variational equations a.e. in (0,T):
\begin{eqnarray}
&&(\frac{\partial c_\epsilon}{\partial t}, \,\varphi) +\eta\, (\nabla
\, c_\epsilon,\, \nabla \varphi)+(\textit{\textbf v}_\epsilon \cdot
\nabla c_l(c_\epsilon, \theta_\epsilon),\, \varphi) = 0,\label{eq33}\\
&&(\frac{\partial \theta_\epsilon}{\partial t},\, \psi) +\kappa\,(\nabla
\theta_\epsilon,\, \nabla \psi) + \rho\, C_p\,(\textit{\textbf v}_\epsilon
\cdot \nabla \theta_\epsilon, \, \psi) = 0,\label{eq34}\\
&&(\frac{\partial \textit{\textbf v}_\epsilon} {\partial t}, \,
\textit{\textbf w}) + \nu \,(\nabla \textit{\textbf v}_\epsilon,
\, \nabla \textit{\textbf w}) +\rho \,(\textit{\textbf v}_\epsilon \cdot
\nabla \textit{\textbf v}_\epsilon,\, \textit{\textbf w})
+(F_i^\epsilon(c_\epsilon,\theta_\epsilon)\,\textit{\textbf v}_\epsilon,
\textit{\textbf w})\hspace*{1.5cm}\nonumber \\
&&\hspace*{1.5cm} =(\mathbf{F}_e(c_\epsilon,\theta_ \epsilon),\,
\textit{\textbf w}),\label{eq35}
\end{eqnarray}
for all $ \ \varphi \in H^1, \ \psi \in H_\theta, \ \textit{\textbf w} \in \textit{\textbf H}_v$.
\end{definition}
\begin{remark}
A weak reproductive solution of the regularized problem (\ref{eq30})-(\ref{eq32c}) is a reproductive regularized 
weak solution of the problem (\ref{eq8})-(\ref{eq10c}). Thus, in what follows, we refer to reproductive regularized 
weak solution of the problem (\ref{eq8})-(\ref{eq10c}) as reproductive weak solution of the problem
(\ref{eq30})-(\ref{eq32c}).
\end{remark}
\begin{remark}
Applying the Maximal Principle to the diffusion and heat equations,
we obtain the inequalities
\begin{eqnarray}
\hspace*{2.5cm} 0\, \leq  & c_\epsilon (x,t)&  \leq \, \gamma_{\ell}(\theta_E)
\quad \mbox{ a.e. in }\mathcal{Q}\,,\label{eq44}\\
\theta_\delta^{\inf} \ \leq \ & \theta_\epsilon(x,t) & \  \leq \
\theta_\delta^{\sup} \quad \mbox{a.e. in }\mathcal{Q}.\label{eq45}
\end{eqnarray}
The proof of (\ref{eq44}) and (\ref{eq45}) is similar to Proposition \ref{Proposition 3.5}.
\end{remark}
\section{Weak reproductive solution of the regularized problem}
In this section, we prove the existence of at least one weak reproductive solution of the problem 
(\ref{eq30})-(\ref{eq32c}), by characterizing it as the limit of a sequence of approximated solutions 
defined on finite dimensional spaces. There follows, we state this result and the proof is done at the 
end of this section.
\begin{theorem}\label{Theorem 5.1} 
Under the hypotheses ({\bf S$_1$}), ({\bf S$_2$}) and ({\bf S$_3$}) the regularized problem (\ref{eq30})-(\ref{eq32c}) 
admits at least one weak reproductive solution.
\end{theorem}
For the sake of simplicity we denote $\tilde \theta_\epsilon = \theta_\epsilon - \theta_\delta$ and $\tilde c_\epsilon 
=c_\epsilon - c_g |\Omega|^{-1}$, then to find a weak reproductive solution of the problem (\ref{eq30})-(\ref{eq32c}) 
is equivalent to: Find the functions $(\,\tilde c_\epsilon,\, \tilde\theta_\epsilon,\, \textit{\textbf{v}}_\epsilon\,) 
\in L^2(H^1)\times L^2(H_\theta) \times L^2(\textit{\textbf{H}}_{ v}\,)$,
such that the following equations are satisfied in the distribution sense in $(0, T)$,
\begin{eqnarray}
&&\hspace*{-.6cm}(\frac{\partial \tilde c_\epsilon}{\partial t},\varphi)+ \eta\,(\nabla \tilde c_\epsilon, 
\nabla \varphi)+(\textit{\textbf v}_\epsilon\cdot \nabla c_l (\textit{\textbf z}_\epsilon),
\varphi) = 0, \label{eq46}\\
&&\hspace*{-.6cm}(\frac{\partial \tilde\theta_\epsilon}{\partial t},\psi)+\kappa\,(\nabla (\tilde\theta_\epsilon 
+\theta_\delta),\nabla \psi)+\rho\, C_p\,(\textit{\textbf v}_\epsilon\cdot\nabla(\tilde\theta_\epsilon 
+\theta_\delta), \psi) = -(\frac{\partial\theta_\delta}{\partial t}, \psi), \label{eq47}\\
&&\hspace*{-.6cm}(\frac{\partial\textit{\textbf v}_\epsilon}{\partial t}, \textit{\textbf w})
+ \nu\,(\nabla \textit{\textbf v}_\epsilon,\nabla \textit{\textbf w})+\rho\, (\textit{\textbf v}_\epsilon
\cdot \nabla \textit{\textbf v}_\epsilon,\, \textit{\textbf w})+(F_i^\epsilon(\textit{\textbf z}_\epsilon)
\,\textit{\textbf v}_\epsilon,\textit{\textbf w})=(\mathbf{F}_e(\textit{\textbf z}_\epsilon),
\textit{\textbf w}), \label{eq48}
\end{eqnarray}
for all $ \ \varphi \in H^1, \ \psi \in H_\theta, \ \textit{\textbf w}\in \textit{\textbf H}_v$, where 
$\textit{\textbf z}_\epsilon=(\tilde c_\epsilon+c_g |\Omega|^{-1}, \tilde\theta_\epsilon + \theta_\delta)$ and
\begin{equation}\label{eq49}
\int_\Omega \tilde c_\epsilon (x,t)\, d x = 0.
\end{equation}
Thus, if we find $(\, \tilde c_\epsilon, \, \tilde \theta_\epsilon,\,\textit{\textbf v}_\epsilon \,)$ reproductive 
solution of (\ref{eq46})-(\ref{eq49}), we have that $(c_\epsilon=\tilde c_\epsilon+ c_g |\Omega|^{-1}, \, 
\theta_\epsilon =\tilde\theta_\epsilon + \theta_\delta, \, \textit{\textbf v}_\epsilon)$ is a weak
reproductive solution of (\ref{eq30})-(\ref{eq32c}).

To prove the existence of at least a weak reproductive solution of the problem (\ref{eq46})-(\ref{eq49}), we consider
the Hilbert orthogonal basis $\{\varphi^i(x)\}^{\infty}_{i=1}$ of $H^1$, $\{\psi^i(x) \}^\infty_{i=1}$ of $H_\theta$ 
and $\{\textit{\textbf w}^i(x)\}^{\infty}_{i=1}$ of $\textit{\textbf H}_{v}$, which are assumed orthonormal in $L^2$
(only to introduce certain simplification in the subsequent treatment).

Let $H^k$ be the subspace of $H^1$ spanned by $\{\varphi^1(x),\ldots, \varphi^k(x)\}$, $H_\theta^k$ be the subspace 
of $H_\theta$ spanned by $\{\psi^1(x),\ldots, \psi^k(x)\}$ and $\textit{\textbf H}_v^k$ be the subspace of
$\textit{\textbf H}_v $ spanned by $\{\textit{\textbf w}^1(x), \ldots, \textit{\textbf w}^k(x)\}$, respectively. 
For every $k \geq 1$, we define approximations $\tilde c^k_\epsilon(x,t)$, $\tilde\theta^k_\epsilon(x,t)$  
and $\textit{\textbf v}^k_\epsilon(x,t)$ of $\tilde c_\epsilon(x,t)$, $\tilde\theta_\epsilon(x,t)$ and
$\textit{\textbf v}_\epsilon(x,t)$ respectively, by means of the following finite expansion:
\begin{eqnarray}
&&\tilde c^k_\epsilon =\sum^k_{i=1} c_{ki}^\epsilon(t)\,\varphi^i(x),
\quad \tilde\theta^k_\epsilon = \sum^k_{i=1}d_{ki}^\epsilon(t)\,
\psi^i(x), \quad \textit{\textbf v}^k_\epsilon= \sum^k_{i=1}
e_{ki}^\epsilon(t)\, \textit{\textbf w}^i(x),\qquad \label{eq50}
\end{eqnarray}
for $ t \in (0,T)$, where the coefficients $c_{ik}^\epsilon (t), \ d_{ik}^\epsilon (t) $ and $e_{ik}^\epsilon (t)$ 
will be calculated in such way that $\tilde c^k_\epsilon, \tilde\theta^k_\epsilon$ and $\textit{\textbf{v}}^k_\epsilon$ 
solve the following approximations of system (\ref{eq46})-(\ref{eq49}):
\begin{eqnarray}
&&(\frac{\partial \tilde c^k_\epsilon}{\partial t}, \varphi^j )
+\eta\,(\nabla  \tilde c^k_\epsilon, \nabla  \varphi^j)
+(\textit{\textbf v}^k \cdot \nabla
c_l(\textit{\textbf z}_\epsilon^k), \varphi^j) = 0, \label{eq51}\\
&&(\frac{\partial\,\tilde\theta^k_\epsilon}{\partial t}, \psi^j)
+\kappa\,(\nabla \tilde\theta^k_\epsilon, \nabla \psi^j)+\rho\,
C_p\,( \textit{\textbf v}^k_\epsilon \cdot \nabla \tilde
\theta^k_\epsilon, \psi^j)+\rho\,C_p\,(\textit{\textbf v}^k_\epsilon
\cdot \nabla \theta_\delta, \psi^j)\nonumber \\
&&\hspace*{2cm}=-\kappa\,(\nabla \theta_\delta, \nabla \psi^j)
- (\frac{\partial\theta_\delta} {\partial t}, \psi^j),\label{eq52}\\
&&(\frac{\partial\,\textit{\textbf v}^k_\epsilon}{\partial t},
\textit{\textbf w}^j)+\nu\, (\nabla \textit{\textbf v}^k_\epsilon,
\nabla \textit{\textbf w}^j)+\rho \,(\textit{\textbf v}^k_\epsilon
\cdot \nabla \textit{\textbf v}^k_\epsilon, \textit{\textbf w}^j)
+(F_i^\epsilon (\textit{\textbf z}_\epsilon^k)\,\textit{\textbf v}^k_\epsilon,
\textit{\textbf w}^j)\nonumber \\
&&\hspace*{2cm}=({\bf F}_e (\textit{\textbf z}_\epsilon^k),\textit{\textbf w}^j),
\label{eq53}
\end{eqnarray}
for all $ \,\varphi^j \in H^k, \ \psi^j \in H_\theta^k, \ \textit{\textbf w}^j \in \textit{\textbf H}_{v}^k$, with
$\textit{\textbf z}_\epsilon^k = (\tilde c^k_\epsilon+ c_g |\Omega|^{-1}, \tilde \theta^k_\epsilon+\theta_\delta)$ and
\begin{eqnarray}
&&\int_\Omega \tilde c^k_\epsilon (x,t)\, dx = 0,\label{eq54}\\
&&\tilde c^k_\epsilon(x,0)= \tilde c_0^k(x), \quad
\tilde\theta^k_\epsilon(x,0) = \tilde \theta_0^k(x), \quad
\textit{\textbf v}^k_\epsilon(x,0)=\textit{\textbf v}_0^k(x)
\qquad {\rm in} \ \Omega,\qquad \label{eq55}
\end{eqnarray}
where
\[\tilde c_0^k(x) \rightarrow \tilde c_0(x)  \mbox{ in } \ L^2, \quad \tilde\theta_0^k(x) \rightarrow 
\tilde\theta_0(x) \mbox{ in } \ L^2,\quad \textit{\textbf v}_0^k(x) \rightarrow
\textit{\textbf v}_0(x) \ \mbox{ in } \ \textit{\textbf H},\]
as $k \rightarrow \infty$ and $\tilde c_0(x)=c_0(x) - c_g |\Omega|^{-1}$ and
$\tilde \theta_0(x)=\theta_0(x)- \theta_\delta (x,0)$.

\begin{remark}\label{remark5.2}
If we denote $ Z^k (x,t) = (\tilde c^k_\epsilon (x,t),\tilde\theta^k_\epsilon(x,t), \textit{\textbf v}^k_\epsilon(x,t))$,
because the orthonormality of $\varphi^i,\psi^i$ and $\textit{\textbf w}^i$ in ${\textbf L}^2$, then
\[\|Z^k (t)\|^2 = \sum_{i=1}^k (\,|c^\epsilon_{ki}(t)|^2
+|d^\epsilon_{ki}(t)|^2+|e^\epsilon_{ki}(t)|^2\,).\]
Since the system (\ref{eq51})-(\ref{eq53}) with the conditions (\ref{eq55}) depends analytically on
\[ {\bf G}(t)=(\, c^\epsilon_{k1}(t), \, \cdots, \, c^\epsilon_{kk}(t), \, d^\epsilon_{k1}(t),\,\cdots,
\, d^\epsilon_{kk}(t),\,e^\epsilon_{k1} (t),\,\cdots,\, e^\epsilon_{kk}(t) \, ),\]
to prove that the approximating solutions $(\tilde c^k_\epsilon,\tilde\theta^k_\epsilon, 
\textit{\textbf v}^k_\epsilon)$ are defined uniquely for all $ t \geq 0$, it is sufficient to prove
that the nonlinear differential system (\ref{eq51})-(\ref{eq53}) with the conditions (\ref{eq55}) has 
a maximal solution ${\bf G}(t)$ defined on some interval $[0,t_k]$. If $t_k < T$ then $\|Z^k (x,t)\|^2$ 
must tend to $ +\infty$ as $t \rightarrow t_k$, but the a priori estimates we shall prove below 
(see Proposition \ref{Proposition 5.3}) show that this is not possible and therefore, $t_k = T$.
\end{remark}
Now, in order to prove that the system (\ref{eq51})-(\ref{eq55}) has at least one reproductive solution, 
previously we get some estimates in the following prepositions.
\begin{proposition}\label{Proposition 5.3} 
Let $(\tilde c_\epsilon^k,\tilde\theta_\epsilon^k, \textit{\textbf v}_\epsilon^k)$ be a 
solution of (\ref{eq51})-(\ref{eq55}). Then, for $ t \in [0,T]$ the following estimates holds
\begin{equation}\label{eq56}
\|Z^k(t)\|^2+\int_0^t (\,\eta\, \|\nabla  \tilde c^k_\epsilon(s)\|^2+\kappa\,\|\nabla 
\tilde\theta^k_\epsilon(s)\|^2 + \nu\, \|\nabla\textit{\textbf v}^k_\epsilon(s)\|^2 \,)\, ds \leq C.
\end{equation}
where $C$ does not depend on $\epsilon$ and $k$.
\end{proposition}
\begin{proof} Considering in (\ref{eq51})-(\ref{eq53}), $\varphi^j=\tilde c^k_\epsilon\,(t),\,\psi^j 
= \tilde\theta^k_\epsilon\,(t)$ and $\textit{\textbf w}^j= \textit{\textbf v}^k_\epsilon\,(t)$, we have
\begin{eqnarray}
&&\frac{1}{2}\frac{d}{dt}\|\tilde c^k_\epsilon(t)\|^2+\eta \,
\|\nabla\tilde c^k_\epsilon(t)\|^2+(\textit{\textbf v}^k_\epsilon (t)
\cdot \nabla c_l(\textit{\textbf z}_\epsilon^k), c^k_\epsilon(t))=0,\label{eq57}\\
&&\frac{1}{2}\frac{d}{dt}\|\tilde\theta^k_\epsilon(t)\|^2
+\kappa\,\|\nabla\tilde\theta^k_\epsilon(t)\|^2+ \rho\, C_p\,
(\textit{\textbf v}^k_\epsilon(t) \cdot\nabla\theta_\delta(t),
\tilde\theta^k_\epsilon(t))\nonumber\\
&&\hspace*{2.2cm}=-\kappa\, (\nabla\theta_\delta(t), \nabla\tilde
\theta^k_\epsilon(t))-(\frac{\partial}{\partial t}\theta_\delta(t),
\tilde\theta^k_\epsilon(t)),\label{eq58}\\
&&\frac{1}{2}\frac{d}{dt}\|\textit{\textbf v}^k_\epsilon(t)\|^2
+\nu\,\|\nabla\textit{\textbf v}^k_\epsilon(t)\|^2+(F_i^\epsilon
(\textit{\textbf z}_\epsilon^k)\,\textit{\textbf v}^k_\epsilon(t),
\textit{\textbf v}^k_\epsilon(t))=({\bf F}_e(\textit{\textbf z}_\epsilon^k),
\textit{\textbf v}^k_\epsilon(t)).\quad\label{eq59}
\end{eqnarray}
Since $\, F_i^\epsilon (\textit{\textbf z}_\epsilon^k) \geq 0$, adding (\ref{eq57})-(\ref{eq59}), we obtain
\begin{eqnarray}\label{eq60}
&&\frac{1}{2}\frac{d}{dt}\|Z^k(t)\|^2+\eta \,
\|\nabla\tilde c^k_\epsilon(t)\|^2+\kappa\, \|\nabla
\tilde\theta^k_\epsilon(t)\|^2+ \nu\,\|\nabla
\textit{\textbf v}^k_\epsilon (t)\|^2 \nonumber \\
&&\hspace*{.5cm}\leq |\,({\bf v}^k_\epsilon(t) \cdot \nabla
c_l(\textit{\textbf z}_\epsilon^k), \tilde c^k_\epsilon(t))\, |
+\rho\, C_p\, |\,(\textit{\textbf v}^k_\epsilon(t)\cdot\nabla
\theta_\delta(t), \tilde\theta^k_\epsilon(t))\,|\nonumber\\
&&\hspace*{1cm} +\kappa\,|\, (\nabla\theta_\delta(t), \nabla
\tilde\theta^k_\epsilon(t))\,|+|\,(\frac{\partial}{\partial t}
\theta_\delta(t),\tilde\theta^k_\epsilon(t))\,|
+|\,({\bf F}_e (\textit{\textbf z}_\epsilon^k),
\textit{\textbf v}^k_\epsilon(t))\,|. \quad 
\end{eqnarray}
Now, by considering ({\bf S$_3$}) together H\"older and Young inequalities, for the first term on the 
right side of inequality (\ref{eq60}), we get
\begin{eqnarray} \label{eq61}
\hspace*{-0.3cm}|(\textit{\textbf v}^k_\epsilon(t)\cdot\nabla c_l(\textit{\textbf z}_\epsilon^k),
\tilde c^k_\epsilon(t))|\nonumber
&=&|(\textit{\textbf v}^k_\epsilon(t) \cdot\nabla\tilde c^k_\epsilon(t),
c_l(\textit{\textbf z}_\epsilon^k))|\nonumber \\
&\leq & C\|\textit{\textbf v}^k_\epsilon(t)\|\|c_l(\textit{\textbf z}_\epsilon^k)\|_{L^\infty}
\|\nabla\tilde c^k_\epsilon(t)\|\nonumber \\
&\leq&\frac{\eta}{2}\,\|\nabla\tilde c^k_\epsilon(t)\|^2+C\,\|\textit{\textbf v}^k_\epsilon(t)\|^2.
\end{eqnarray}
By taking into account (\ref{eq45}) and, using the H\"older and Young inequalities, for the others terms on 
the right side of (\ref{eq60}) we have
\begin{eqnarray}
\rho\, C_p\, |\, (\textit{\textbf v}^k_\epsilon (t)\cdot\nabla
\,\theta_\delta(t), \tilde\theta^k_\epsilon(t))\, |&\leq & C\,
\|\textit{\textbf v}^k_\epsilon (t)\|_{L^6}\|\nabla
\theta_\delta(t)\|\,\|\tilde \theta^k_\epsilon(t)\|_{L^3}\nonumber \\
&\leq &\,\frac{\nu}{2}\,\| \nabla \textit{\textbf v}^k_\epsilon (t)\|^2
+ C\, \|\nabla \theta_\delta(t)\|^2 ,\label{eq62}\\
&&\nonumber\\
\kappa\,|\,(\nabla \theta_\delta(t), \nabla \tilde\theta^k_\epsilon(t))\,|
&\leq &\, C\, \| \nabla \theta_\delta(t)\|^2 +\frac{\kappa}{2}\,\|\nabla
\,\tilde\theta^k_\epsilon(t)\|^2,\label{eq63}\\
&&\nonumber\\
|\,(\frac{\partial}{\partial t} \theta_\delta(t),\tilde\theta^k_\epsilon(t))\,|
&\leq & C\,\|\frac{\partial}{\partial t} \theta_\delta(t)\|^2
+ C\, \|\tilde\theta^k_\epsilon(t)\|^2,\label{eq64}
\end{eqnarray}
considering (\ref{eq45}), ({\bf S$_3$}) and since ${\bf g} \in L^\infty$, we get
\begin{eqnarray}\label{eq66}
|({\bf F}_e(\textit{\textbf z}_\epsilon^k),\textit{\textbf v}^k_\epsilon(t))|
&\leq &C\, \|{\bf g}\|\,\|\tilde\theta^k_\epsilon(t) +
\theta_\delta(t)\|_{L^3}\,\|\textit{\textbf v}^k_\epsilon(t)\|_{L^6}
+C\,\|{\bf g}\|\,\|\textit{\textbf v}^k_\epsilon(t)\|\nonumber \\
&&+ C\, \|{\bf g}\|\,\|c_l(\textit{\textbf z}_\epsilon^k)\|_{L^\infty}
\|\textit{\textbf v}^k_\epsilon(t)\|\nonumber\\
&\leq & C+\frac{\nu}{4}\, \|\nabla \textit{\textbf v}^k_\epsilon(t)\|^2.
\end{eqnarray}
Then, carrying inequalities (\ref{eq61})-(\ref{eq66}) in (\ref{eq60}), we have the inequality
\begin{eqnarray}\label{eq67}
&&\frac{1}{2}\frac{d}{dt}\| Z^k(t)\|^2 + \frac{\eta}{2}\,\|\nabla\tilde c^k_\epsilon(t)\|^2
+\frac{\kappa}{2}\,\|\nabla\tilde\theta^k_\epsilon(t)\|^2+\frac{\nu}{2}\,
\|\nabla\textit{\textbf v}^k_\epsilon(t)\|^2\nonumber \\
&&\hspace*{.5cm}\leq C\,\|Z^k(t)\|^2+C\,\|\theta_\delta(t)\|^2_{H^1}
+ C\,\|\frac{\partial}{\partial t} \theta_\delta(t)\|^2 + C.
\end{eqnarray}
Now, taking into account ({\bf S$_2$}), and by integrating (\ref{eq67}) from $0$ to $t \ (t \in [0,T])$, we obtain
\begin{eqnarray*}
&&\|Z^k(t)\|^2+\int_0^t (\eta\,\|\nabla  \tilde c^k_\epsilon(s)\|^2+\kappa\,\|\nabla \tilde
\theta^k_\epsilon(s)\|^2+\nu\, \|\nabla\textit{\textbf v}^k_\epsilon(s)\|^2 )\, ds\\
&&\qquad \leq C + \|Z^k(0)\|^2 + C \int_0^t\|Z^k(s)\|^2 ds,
\end{eqnarray*}
thus the Gronwall's inequality (Lemma \ref{Lemma 3.2}), implies
\begin{equation}\label{eq68}
\hspace*{-.5cm}\|Z^k(t)\|^2+ \int_0^t (\eta\,\|\nabla \tilde c^k_\epsilon(s)\|^2
+\kappa\,\|\nabla  \tilde\theta^k_\epsilon(s)\|^2+\nu\, \|\nabla
\textit{\textbf v}^k_\epsilon(s)\|^2 )\, ds  \leq C(T),
\end{equation}
and consequently inequality (\ref{eq56}). 
\end{proof}
\begin{proposition}\label{Proposition 5.4}  
Let $(\tilde c_\epsilon^k, \tilde\theta_\epsilon^k, \textit{\textbf v}_\epsilon^k)$ be a solution of
(\ref{eq51})-(\ref{eq55}). Then it is satisfied the estimate
\begin{equation}\label{eq69}
\mbox{\large e}\, ^{\beta T} \, \| Z^k(T)\|^2 \leq \|Z^k(0)\|^2 + C(\beta,T),
\end{equation}
where $\beta = C_\Omega^{-1}\min \{ \eta, \kappa, \nu \}$.
\end{proposition}
\begin{proof} Using the Poincare's inequality ($\|Z^k(t)\|^2 \leq C_\Omega \|\nabla  Z^k(t)\|^2$) and the
definition of $\beta$, from (\ref{eq67}) and (\ref{eq68}), we have
\[\frac{d}{dt}\|Z^k(t)\|^2+\beta \|Z^k(t)\|^2 \leq C\,(T) + C\,\|\theta_\delta(t)\|^2_{H^1} 
+ C\, \| \frac{\partial}{\partial t}\theta_\delta(t)\|^2.\]
Consequently, for all $ t \in [0,T]$
\begin{eqnarray}\label{eq70}
\frac{d}{dt}( \, \mbox{\large e}\,^{\beta t} \,\|Z^k(t)\|^2 )&\leq & \mbox{\large e}\,^{\beta T} \, (\, C(T) 
+ C\,\|\theta_\delta(t) \|^2_{H^1} + C\, \| \frac{\partial}{\partial t}\theta_\delta(t)\|^2\,).
\end{eqnarray}
Now, taking into account ({\bf S$_3$}) ($\theta_\delta \in H^1(H^1)$) and integrating (\ref{eq70}) from $0$ to
$T$, we obtain inequality (\ref{eq69}). 
\end{proof}

In the following Lemma we will prove the existence of reproductive
solution of the approximated problem (\ref{eq51})-(\ref{eq55}).
\begin{lemma}\label{Lemma 5.5} 
Under the hypotheses of Theorem \ref{Theorem 5.1}, the problem (\ref{eq51})-(\ref{eq55}) admits at least 
one reproductive solution $ (\tilde c_\epsilon^k,\tilde\theta_\epsilon^k,$ $\textit{\textbf v}_\epsilon^k),
\ \forall \, k \geq 1$. Furthermore, there exists $(\tilde c_\epsilon, \tilde \theta_\epsilon, 
\textit{\textbf v}_\epsilon)$ such that as $k \rightarrow \infty$,
\begin{eqnarray}
\tilde c^k_\epsilon &\rightarrow &\tilde c_\epsilon \quad \mbox{weakly in } L^2(H^1),\quad 
\mbox{weakly -* in } L^\infty(L^2),\label{eq71}\\
\tilde\theta^k_\epsilon & \rightarrow &\tilde\theta_\epsilon\quad\mbox{ weakly in } L^2(H_\theta),\quad 
\mbox{weakly -* in } L^\infty(L^2),\label{eq72}\\
\textit{\textbf v}^k_\epsilon & \rightarrow &\textit{\textbf v}_\epsilon \quad \mbox{ weakly in } 
L^2(\textit{\textbf H}_{v}), \quad\mbox{weakly -* in } L^\infty(\textit{\textbf H}).\label{eq73}
\end{eqnarray}
\end{lemma}
\begin{proof} Let be the operator $\mathbf{Y}:[0,T] \mapsto \mathbb{R}$ defined as
\[\mathbf{Y}(t)=(\, c_{k1}^\epsilon(t),\,\cdots, \, c_{kk}^\epsilon(t), \, d_{k1}^\epsilon(t), \, 
\cdots, \, d_{kk}^\epsilon(t), \, e_{k1}^\epsilon(t), \, \cdots, \, e_{kk}^\epsilon(t) \, )\]
where the time dependent functions $ \{ c_{ki}^\epsilon(t),\, d_{ki}^\epsilon(t),\,  e_{ki}^\epsilon(t)\}_{i=1}^k$ 
are the coefficients of the expansion $Z^k \,(x,t)$ given in (\ref{eq50}).
Thus, taking into account Remark \ref{remark5.2}, we have
\begin{equation}\label{eq74}
\|{\bf Y}(t)\|_{\mathbb{R}^{3k}} = \|Z^k(t)\|, \quad \forall \, t\in [0,T].
\end{equation}
We define the operator
\begin{eqnarray*}
 \ \Phi \, : \ \mathbb{R}^{3k} &\longmapsto  &  \mathbb{R}^{3k} \\
\mathbf{r} \ &\longmapsto & \Phi(\mathbf{r}) = \mathbf{Y}(T) \qquad \qquad \qquad
\end{eqnarray*}
where for each $k \geq 1$ and $\mathbf{r} =(r_1, r_2, \dots, r_{3k})$, $\mathbf{Y}(T)$ is the vector of 
coefficients of the solution of the problem (\ref{eq51})-(\ref{eq55}) at time $T$, with initial condition given by
\[\tilde c^k_0(x) = \sum_{i=1}^k r_i \, \varphi^i(x), \ \quad\tilde \theta^k_0(x) = \sum_{i=1}^k r_{k+i} 
\, \psi^i(x), \ \quad\textit{\textbf v}^k_0(x) = \sum_{i=1}^k r_{2k+i} \,\textit{\textbf w}^i(x).\] 
We will prove that $\Phi$ has at least one fixed point, as consequence of Leray-Schauder's homotopy
Theorem. For this purpose, it is sufficient to show that for any $\lambda \in [0,1]$, a solution of the equation
\begin{equation}\label{eq75}
\lambda \, \Phi(\, \mathbf{r}(\lambda)\,) = {\bf r}(\lambda),
\end{equation}
is bounded independently of $\lambda$. Since ${\bf r}(0) = {\bf 0}$, the proof will be restrict to $\lambda \in (0,1]$.

Since $\lambda \leq1$, by the definition of $\Phi$ together equalities (\ref{eq74}) and (\ref{eq75}), we have
\begin{equation}\label{eq76}
\|\mathbf{r}(\lambda)\|^2_{\mathbb{R}^{3k}} \leq \|\Phi(\mathbf{r}(\lambda))\|^2_{\mathbb{R}^{3k}}\leq \|Z^k(T)\|^2.
\end{equation}
Then, from (\ref{eq76}) and  (\ref{eq69}), we obtain
\[ \mbox{\large e}\, ^{\beta T}\, \|\mathbf{r}(\lambda)\|^2_{\mathbb{R}^{3k}}\leq \mbox{\large e}\, ^{\beta T}\, 
\|Z^k(T)\|^2\leq \|\,\mathbf{r}(\lambda)\|^2_{\mathbb{R}^{3k}}+ C(\beta,T)\]
thus $ \ \displaystyle (\mbox{\large e}\, ^{\beta T}-1) \, \|\, {\bf r}(\lambda)\,\|^2_{\mathbb{R}^{3k}} 
\leq C(\beta,T)$,
and consequently
\begin{equation}\label{eq77}
\| \, {\bf r}(\lambda)\,\|^2_{\mathbb{R}^{3k}} \leq  \frac{1}{\mbox{\large e}\, ^{\beta T}-1} \, 
C(\beta,T), \ \forall \, \lambda \in [0,1].
\end{equation}
Since the upper bound in (\ref{eq77}) is independent on $\lambda\in [0,1]$, we have established that 
operator $\Phi$ has at least one fixed point, and from (\ref{eq75}) the fixed point of $\Phi$
is denoted by ${\bf r}(1)$ and satisfies (\ref{eq77}). Therefore, we have proved the result of the 
Lemma \ref{Lemma 5.5}. Moreover, (\ref{eq71})-(\ref{eq73}) are a consequence of Proposition \ref{Proposition 5.3}.
\end{proof}

Nevertheless, the convergences given in Lemma \ref{Lemma 5.5} are insufficient to prove that the limit 
functions become a reproductive weak solution of the  problem (\ref{eq30})-(\ref{eq32c}). Therefore, we need
to prove strong convergence.
\begin{lemma}\label{Lemma 5.6} 
Under the hypotheses of Lemma \ref{Lemma 5.5}, as $k \rightarrow \infty,$ the sequences $\{\tilde
c_\epsilon^k\}, \{\tilde \theta_\epsilon^k\}$ and $\{\textit{\textbf v}_\epsilon^k\}$ converge weakly in
$L^2(\Omega)$, uniformly for $t \in [0,T]$.
\end{lemma}

\begin{proof} For fixed $j$ and $k \geq j$, we will prove that $\,\xi_{kj}(t)=(\tilde c_\epsilon^k(t),\varphi^j), 
\ \chi_{kj}(t)=(\tilde\theta_\epsilon^k(t),\psi^j)$ and $\zeta_{kj}(t)=(\textit{\textbf v}_\epsilon^k(t), 
\textit{\textbf w}^j)$ are uniformly bounded and equicontinuous families of functions on $[0,T]$, and then 
applying the Theorem of Ascoli-Arzel\`a, we obtain the uniform convergence on $[0, T]$.

The uniform boundedness of the family $\{\xi_{kj}(t), \, \chi_{kj}(t), \, \zeta_{kj}(t)\}$ follows from 
Proposition \ref{Proposition 5.3}, therefore, we are going to prove the equicontinuity.

The equalities (\ref{eq51})-(\ref{eq53}), can be rewritten in the form
\begin{eqnarray}
\xi_{kj}^{\prime}(t)&=&-\eta \,( \nabla  \tilde c^k_\epsilon,
\nabla  \varphi^j)
+(  \textit{\textbf v}^k_\epsilon
 \cdot\nabla c_l(\textit{\textbf z}_\epsilon^k), \varphi^j)
,\label{eq78}\\
\chi_{kj}^{\prime}(t)&=& -\kappa\,(\nabla  \tilde\theta^k_\epsilon,
\nabla \psi^j )- \rho\, C_p\, (\textit{\textbf v}^k_\epsilon \cdot
\nabla \tilde\theta^k_\epsilon , \psi^j )-\rho\,C_p\,(\textit{\textbf
v}^k_\epsilon \cdot \nabla \theta_\delta, \psi^j )\nonumber \\
&& - \kappa \, ( \nabla \theta_\delta, \nabla  \psi^j)-(\frac{
\partial}{\partial t} \theta_\delta, \psi^j),\label{eq79}\\
\zeta_{kj}^{\prime} (t)&=& -\nu\, (\nabla \textit{\textbf v}^k_\epsilon,
\nabla \textit{\textbf w}^j)- \rho \, (\textit{\textbf v}^k_\epsilon
\cdot \nabla \textit{\textbf v}^k_\epsilon, \textit{\textbf w}^j)
- (F_i^\epsilon (\textit{\textbf z}_\epsilon^k)\,\textit{\textbf v}^k_\epsilon,
\textit{\textbf w}^j)\nonumber \\
&&+({\bf F}_e (\textit{\textbf z}_\epsilon^k), \textit{\textbf w}^j),\label{eq80}
\end{eqnarray}
for $j=1,\dots,k$.

By integrating (\ref{eq78})-(\ref{eq80}) with respect to $t$ from $t$ to $t+\Delta t$ and estimating 
the right-hand side  by using Schwarz's inequality, we obtain
\begin{eqnarray}\label{ant}
&&|\xi_{kj}(t+\Delta t)-\xi_{kj}(t)|\leq\eta \,\int_t^{t+\Delta t}  \vert\nabla  \tilde c^k_\epsilon\cdot
\nabla  \varphi^j\vert\,dt+\int_t^{t+\Delta t}   \vert\textit{\textbf v}^k_\epsilon
 c_l(\textit{\textbf z}_\epsilon^k)\cdot\nabla \varphi^j\vert\,dt\nonumber \\
&&\hspace*{.5cm}\leq \eta \int_t^{t+\Delta t} \|\nabla \tilde c_\epsilon^k(t)\|\,\|\nabla \varphi^j\| dt 
+ \int_t^{t+ \Delta t}\|\textit{\textbf v}^k_\epsilon(t)\|\,\|\nabla \varphi^j\|\,
\|c_l (\textit{\textbf z}_\epsilon^k)\|_{L^\infty} dt \nonumber \\
&&\hspace*{.5cm}\leq C\,\int_t^{t+\Delta t} \|\nabla\,\tilde c_\epsilon^k(t)\|\,dt +C\, \int_t^{t+\Delta t}
\|\textit{\textbf v}^k_\epsilon(t)\|\, dt \nonumber \\
&&\hspace*{.5cm}\leq C\sqrt{\Delta t}\, \{(\int_t^{t+\Delta t} \|\nabla \tilde c_\epsilon^k (t)\|^2 dt)^{1/2}
+(\int_t^{t+\Delta t}\|\textit{\textbf v}^k_\epsilon(t)\|^2 dt)^{1/2}\}.\label{eq81}
\end{eqnarray}
The terms $(\textit{\textbf v}^k_\epsilon \cdot\nabla \tilde\theta^k_\epsilon , \psi^j )$ and $(\textit{\textbf
v}^k_\epsilon \cdot \nabla \theta_\delta, \psi^j )$ of (\ref{eq79}) can be bounded, after integrating by 
parts, as $(  \textit{\textbf v}^k_\epsilon\cdot\nabla c_l(\textit{\textbf z}_\epsilon^k), \varphi^j)$ 
in (\ref{ant}) or without integrating as follows:
\begin{eqnarray}
&&|\chi_{kj}(t+\Delta t)-\chi_{kj}(t)|\nonumber \\
&&\hspace*{.5cm}\leq C\, \int_t^{t+\Delta t} \|\nabla\,\tilde \theta_\epsilon^k(t)\|\,\|\nabla \psi^j\|\, dt 
+ C\, \int_t^{t+\Delta t}\|\textit{\textbf v}^k_\epsilon (t)\|_{L^3}\,\|\nabla \tilde \theta_\epsilon^k\|
\,\| \psi^j\|_{L^6}\, dt\nonumber \\
&&\hspace*{1cm}+C\,\int_t^{t+\Delta t} \|\textit{\textbf v}^k_\epsilon
(t) \|_{L^3}\,\|\nabla\theta_\delta(t)\| \,\| \psi^j\|_{L^6}\,dt\nonumber \\
&&\hspace*{1cm}+ C\,\int_t^{t+\Delta t}(\,\|\delta\theta_\delta(t)\|\,\|\nabla \psi^j\|
+\|\frac{\partial}{\partial t}\theta_\delta(t)\|\,\| \psi^j\|\,) \,dt \nonumber \\
&&\hspace*{.5cm}\leq C\,\sqrt{\Delta t} (\int_t^{t+\Delta t}
\|\nabla \tilde \theta_\epsilon^k(t)\|^2 dt)^{1/2}\nonumber\\
&&\hspace*{1cm}+C \,\sqrt[4]{\Delta t} (\int_t^{t+\Delta t}
 \|\nabla \textit{\textbf v}^k_\epsilon(t)\|^2 dt)^{1/4}(\int_t^{t+\Delta t}
\|\nabla \tilde \theta_\epsilon^k(t)\|^2 dt)^{1/2}\nonumber\\
&&\hspace*{1cm}+C \,\sqrt[4]{\Delta t} (\int_t^{t+\Delta t}\|\nabla \textit{\textbf v}^k_\epsilon(t)\|^2 
dt)^{1/4}(\int_t^{t+\Delta t}\| \tilde \theta_\delta(t)\|^2_{H^1} dt)^{1/2}\nonumber\\
&&\hspace*{1cm}+C\sqrt{\Delta t}\, \Big( (\int_t^{t+\Delta t} \|\theta_\delta(t)\|^2_{H^1}dt)^{1/2}
+(\int_t^{t+\Delta t}\|\frac{\partial}{\partial t} \theta_\delta(t)\|^2 dt)^{1/2}\Big).\label{eq82}
\end{eqnarray}
Here, we have used that $\|\textit{\textbf v}^k_\epsilon (t)\|_{L^3}\leq\|\textit{\textbf v}^k_\epsilon (t)\|^{1/2}
\,\|\nabla\textit{\textbf v}^k_\epsilon (t)\|^{1/2}$ and $\|\textit{\textbf v}^k_\epsilon (t)\|\leq C$.
\begin{eqnarray}\label{eq83}
&&|\zeta_{kj}(t+\Delta t)-\zeta_{kj}(t)|\nonumber \\
&&\hspace*{.5cm}\leq C\, \int_t^{t+\Delta t} \|\nabla \textit{\textbf v}^k_\epsilon(t)
\|\, dt +C\, \int_t^{t+\Delta t}\, \|\nabla \textit{\textbf v}^k_\epsilon(t)\|^{3/2}\,dt\nonumber\\
&&\hspace*{1cm} +C\,\int_t^{t+\Delta t}\|F_i^\epsilon(\textit{\textbf z}_\epsilon^k)
\|_{L^\infty}\,\|\textit{\textbf v}^k_\epsilon(t)\|\, dt
+\,\int_t^{t+\Delta t} \|\textit{\textbf F}_e(\textit{\textbf z}_\epsilon^k)\|\,dt\nonumber \\
&&\hspace*{.5cm}\leq C\,\sqrt{\Delta t}\,(\int_t^{t+\Delta t} \|\nabla
\, \textit{\textbf v}^k_\epsilon(t)\|^2 dt)^{1/2}+C\,\sqrt[4]{\Delta t}\,(\int_t^{t+\Delta t} \|\nabla
\, \textit{\textbf v}^k_\epsilon(t)\|^2)^{3/4} dt\nonumber\\
&&\hspace*{1cm} +C\, (\max_t\|\textit{\textbf v}^k_\epsilon(t)\|)\,\sqrt{\Delta t}\,(\int_t^{t+\Delta t}
\|F_i^\epsilon(\textit{\textbf z}_\epsilon^k)\|_{L^\infty}^2 dt)^{1/2}+ C \,\Delta t.
\end{eqnarray}
Then by Lemma \ref{Lemma 5.5}, we have that the right-hand side of inequalities (\ref{eq81}), (\ref{eq82}) 
and (\ref{eq83}) converges to zero uniformly in $k$ as $\Delta t \rightarrow 0$ and it is follow that 
$\,\xi_{kj}(t), \, \chi_{kj}(t)$ and $\zeta_{kj}(t)$ are equicontinuous families on $[0,T]$. Thus,
from Lemma \ref{Lemma 5.5} and the Theorem of Arzel\`a-Ascoli the proof of Lemma \ref{Lemma 5.6} is complete.
\end{proof}

In the following lemma, we will prove strong convergence of the
reproductive solutions of the problem (\ref{eq51})--(\ref{eq55}).
\begin{lemma}\label{Lemma 5.7} 
Under the hypotheses of Lemma \ref{Lemma 5.6}, the sequence $\{(\tilde c^k_\epsilon,
\tilde\theta^k_\epsilon, \textit{\textbf v}^k_\epsilon)\}$ converges strongly in $L^2(L^2)\times 
L^2(L^2) \times L^2(\textit{\textbf H})$ toward $(\tilde c_\epsilon, \,\tilde\theta_\epsilon, \,
\textit{\textbf v}_\epsilon)$.
\end{lemma}
\begin{proof} By setting
$$\textit{\textbf u}(t)=(c_\epsilon^{m,n}(t), \theta_\epsilon^{m,n}(t),\textit{\textbf v}_\epsilon^{m,n}(t))$$ 
in the inequality of Friedrich, Lemma \ref{Lemma 3.3}, being
\begin{eqnarray*}
c_\epsilon^{m,n}(t)&=&\tilde c^{k_m}_\epsilon(t)-\tilde c^{k_n}_\epsilon(t), \quad \theta_\epsilon^{m,n}(t)
\ = \ \tilde\theta^{k_m}_\epsilon (t)-\tilde\theta^{k_n}_\epsilon(t),\\ 
\textit{\textbf v}_\epsilon^{m,n}(t)
&=&\textit{\textbf v}^{k_m}_\epsilon(t)-\textit{\textbf v}^{k_n}_\epsilon(t),
\end{eqnarray*}
and integrating it with respect to $t$ from $0$ to $T$,  we obtain
\begin{eqnarray}\label{eq85}
&&\int_0^T (\|c_\epsilon^{m,n}(t)\|^2 +\|\theta_\epsilon^{m,n}(t)\|^2+
 \|\textit{\textbf v}_\epsilon^{m,n}(t)\|^2 )dt \nonumber\\
&&\hspace*{.5cm}\leq \sum_{i=1}^{N_\omega}\int_0^T [(c_\epsilon^{m,n}(t), \varphi^i)^2
+(\theta_\epsilon^{m,n}(t), \psi^i)^2+(\textit{\textbf v}_\epsilon^{m,n}(t),
 \textit{\textbf w}^i)^2] dt \nonumber \\
&&\hspace*{1cm}+\omega \, \int_0^T (\|\nabla  c_\epsilon^{m,n}(t)\|^2
+\|\nabla \theta_\epsilon^{m,n}(t)\|^2+\|\nabla
(\textit{\textbf v}_\epsilon^{m,n}(t)\|^2 ) dt,
\end{eqnarray}
for any given $\omega >0$.

From Proposition \ref{Proposition 5.3}, the last integral on the right-hand side of (\ref{eq85}) does 
not exceed a fixed constant for any $k_m$ and $k_n$, and from Lemma \ref{Lemma 5.6}, the first 
integral can be made arbitrarily small for sufficiently large $k_m$ and $k_n$ because of the uniform convergence in $t$.
Therefore, for sufficiently large $k_m$ and $k_n$
\[ \int_0^T (\|\tilde c^{k_m}_\epsilon(t)-\tilde c^{k_n}_\epsilon(t)\|^2+\|\tilde\theta^{k_m}_\epsilon(t)
-\tilde\theta^{k_n}_\epsilon(t)\|^2+\|\textit{\textbf v}^{k_m}_\epsilon(t)
-\textit{\textbf v}^{k_n}_\epsilon(t)\|^2)dt \rightarrow 0,\]
and this implies the Lemma \ref{Lemma 5.7}. Moreover, from Lemma \ref{Lemma 5.5}, we conclude

\noindent $\tilde c_\epsilon^k \rightarrow \tilde c_\epsilon,\quad \tilde\theta_\epsilon^k \rightarrow 
\tilde\theta_\epsilon\mbox{ strongly in } L^2(L^2), \mbox{ and } \textit{\textbf v}_\epsilon^k
\rightarrow \textit{\textbf v}_\epsilon \mbox{ strongly in } L^2(\textit{\textbf H})$. 
\end{proof}
The results of the Lemma \ref{Lemma 5.5} and Lemma \ref{Lemma 5.7} enable us to prove the Theorem \ref{Theorem 5.1}.
\vspace{3mm}

\noindent{\bf Proof of the Theorem \ref{Theorem 5.1}}
\vspace{3mm}

We going to prove that $( \,\tilde c_\epsilon\,(x,t),\, \tilde\theta_\epsilon\,(x,t), \, 
\textit{\textbf{v}}_\epsilon\,(x,t) \,)$, obtained in Lemma \ref{Lemma 5.7}, is a reproductive 
weak solution of problem (\ref{eq30})-(\ref{eq32c}) taking limit as $k \rightarrow \infty$ in
(\ref{eq51})-(\ref{eq54}) after integrating in $[0,T]$.

In fact, considering $\phi \in C^\infty([0,T])$ such that $\phi(T)=0$, and using by parts integration, we have
\begin{eqnarray*}
\int_0^T (\frac{\partial}{\partial t}(\tilde c^k_\epsilon-\tilde c_\epsilon)(t), \varphi^j \phi(t))dt 
&= &-\int_0^T (\tilde c^k_\epsilon(t) -\tilde c_\epsilon(t), \varphi^j \, \phi^{'}(t))dt\\
&&- (\tilde c^k_\epsilon(x,0)-\tilde c_\epsilon(x,0), \varphi^j\phi(0)),
\end{eqnarray*}
then, from (\ref{eq30b}), (\ref{eq55}) and Lemma \ref{Lemma 5.7}, we obtain that as $\ k \rightarrow \infty$
\begin{equation}\label{eq86}
\int_0^T (\frac{\partial}{\partial t}(\tilde c^k_\epsilon-\tilde
c_\epsilon)(t), \varphi^j \phi(t))dt \longrightarrow 0.
\end{equation}
Similarly, if $ \ k \rightarrow \infty $ we have
\begin{eqnarray}
&&\int_0^T (\frac{\partial}{\partial t}(\tilde\theta^k_\epsilon-\tilde\theta_\epsilon)(t), \,\psi^j
\phi(t))dt \longrightarrow 0, \label{87} \\
&&\int_0^T (\frac{\partial}{\partial t}(\textit{\textbf v}^k_\epsilon-\textit{\textbf v}_\epsilon)(t),
 \,\textit{\textbf w}^j \phi(t))dt \longrightarrow 0.\label{eq88}
\end{eqnarray}
Also, by Lemma \ref{Lemma 5.5} when $ \ k \rightarrow \infty $, we conclude
\begin{eqnarray}
&&\int_0^T (\nabla \tilde c^k_\epsilon(t)- \nabla  \tilde c_\epsilon(t), \nabla \varphi^j \phi(t))dt 
\longrightarrow 0, \label{eq89}\\
&&\int_0^T (\nabla  \tilde\theta^k_\epsilon (t)- \nabla \tilde\theta_\epsilon(t),\nabla \psi^j \phi(t))dt
\longrightarrow 0,\label{eq90}\\
&&\int_0^T (\nabla  \textit{\textbf v}^k_\epsilon(t)- \nabla\,\textit{\textbf v}_\epsilon(t), \nabla 
\textit{\textbf w}^j \phi(t))dt \longrightarrow 0.\label{eq91}
\end{eqnarray}
To bound the nonlinear terms, we consider $b_\epsilon$ that will represent $c_l (\textit{\textbf z}_\epsilon^k)$, 
$\tilde\theta^k_\epsilon$ or $\textit{\textbf v}_\epsilon^k$ and $\xi^j$ that will be $\phi^j$, $\psi^j$ 
or $\textit{\textbf w}^j$ in the following expression:
\begin{eqnarray*}
&&|(\textit{\textbf v}^k_\epsilon \cdot \nabla b^k_\epsilon
-\textit{\textbf v}_\epsilon \cdot \nabla b_\epsilon ,\xi^j\,\phi(t))|\\
&&\hspace*{.5cm}\leq |((\textit{\textbf v}_\epsilon -\textit{\textbf v}^k_\epsilon )\cdot 
\nabla  \xi^j, b^k_\epsilon)-(\textit{\textbf v}_\epsilon \cdot \nabla \xi^j,(b^k_\epsilon
-b_\epsilon ))|\,\sup_{t\in [0,T]}|\phi(t)|\\
&&\hspace*{.5cm} \leq  C \,(\|\textit{\textbf v}^k_\epsilon-\textit{\textbf v}_\epsilon \|\, 
\|b^k_\epsilon \|_{L^6} +\|\textit{\textbf v}_\epsilon \|_{L^6}\, \|b^k_\epsilon-b_\epsilon \|)\,
\|\nabla \xi^j\|_{L^3}\\
&&\hspace*{.5cm}\leq  C \, \|\textit{\textbf v}^k_\epsilon-\textit{\textbf v}_\epsilon \|\, 
\|\nabla b^k_\epsilon \| + C\,\|\nabla \textit{\textbf v}_\epsilon \|\, \|b^k_\epsilon-b_\epsilon \|.
\end{eqnarray*}
Then, by integrating with respect $t$ from $0$ to $T$ and using the Schwarz's inequality, we obtain
\begin{eqnarray}\label{eq92}
&&\int_0^T(\textit{\textbf v}^k_\epsilon \cdot \nabla\,b^k_\epsilon-\textit{\textbf v}_\epsilon 
\cdot \nabla b_\epsilon,\,\xi^j)\phi(t)\,dt \nonumber \\
&& \qquad \leq C\, (\int_0^T\|\textit{\textbf v}^k_\epsilon\,(t)-\textit{\textbf v}_\epsilon \,(t)\|^2 
dt)^{1/2}\,(\int_0^T\|\nabla b^k_\epsilon\,(t) \|^2 dt )^{1/2}\nonumber\\
&&\qquad \quad + C\, (\int_0^T\|\nabla \textit{\textbf v}_\epsilon\,(t)\|^2 dt )^{1/2}\,
(\int_0^T \|b^k_\epsilon\,(t)-b_\epsilon\,(t)\|^2 dt)^{1/2}.
\end{eqnarray}
Thus, from Proposition \ref{Proposition 5.3}, the continuity of $c_l(\cdot, \cdot)$, Lemma \ref{Lemma 5.7} 
and inequality (\ref{eq92}), as $\ k \rightarrow \infty$ we have
\begin{eqnarray}
&&\int_0^T(\textit{\textbf v}^k_\epsilon\,(t)  \cdot \nabla c_{l \epsilon}^k\,(t) 
- \textit{\textbf v}_\epsilon\,(t)\cdot\nabla  c_{l \epsilon}\,(t), \varphi^j)\phi(t)dt
\longrightarrow 0,\label{eq93} \\
&&\int_0^T(\textit{\textbf v}^k_\epsilon\,(t)  \cdot \nabla\tilde\theta^k_\epsilon\,(t) 
- {\bf v}_\epsilon\,(t) \cdot \nabla\, \tilde\theta_\epsilon\,(t) , \psi^j)\phi(t)dt
\longrightarrow 0, \label{eq94}\\
&&\int_0^T(\textit{\textbf v}^k_\epsilon\,(t)  \cdot \nabla\textit{\textbf v}^k_\epsilon\,(t) 
-\textit{\textbf v}_\epsilon\,(t)\cdot \nabla \textit{\textbf v}_\epsilon\,(t),
 \textit{\textbf w}^j)\phi(t)dt \longrightarrow 0.\label{eq95}
\end{eqnarray}
For the external force, taking into account that
\[{\bf F}_e(\textit{\textbf z}_\epsilon^k)- {\bf F}_e(\textit{\textbf z}_\epsilon)
= \rho\,{\bf g} [\alpha\,(\tilde \theta_ \epsilon^k-\tilde\theta_ \epsilon) 
+\beta\, (\tilde c_\epsilon^k- \tilde c_\epsilon)],\]
we have
\begin{eqnarray*}
&&|({\bf F}_e(\textit{\textbf z}_\epsilon^k)- {\bf F}_e(\textit{\textbf z}_\epsilon),
\textit{\textbf w}^j\, \phi(t))|\\
&&\hspace*{.5cm}\leq C\,\|{\bf g}\|_\infty(\,\| \tilde \theta_\epsilon^k-\tilde\theta_ \epsilon\|
+\|\tilde c_\epsilon^k - \tilde c_\epsilon\|)\|\textit{\textbf w}^j\|\,\sup_{t\in [0,T]}|\phi(t)| \\
&&\hspace*{.5cm}\leq  C\,(\|\tilde\theta_\epsilon^k-\tilde\theta_\epsilon\|+ \|\tilde c_\epsilon^k
-\tilde c_\epsilon\|),
\end{eqnarray*}
and, by integrating and using the Schwarz's inequality, we get
\begin{eqnarray}\label{eq96}
&&\int_0^T({\bf F}_e(\textit{\textbf z}_\epsilon^k)- {\bf F}_e(\textit{\textbf z}_\epsilon),
\textit{\textbf w}^j)\,\phi(t) \,dt \nonumber \\
&&\hspace*{1cm}\leq  C\,(\int_0^T \|\tilde\theta_ \epsilon^k -\tilde\theta_\epsilon \|^2 dt)^{1/2} 
+C\, (\int_0^T\|\tilde c_\epsilon^k-\tilde c_\epsilon\|^2 dt)^{1/2}.
\end{eqnarray}
Moreover,
\begin{eqnarray*}
&&F_i^\epsilon(\textit{\textbf z}_\epsilon^k)\, \textit{\textbf v}_\epsilon^k -
F_i^\epsilon(\textit{\textbf z}_\epsilon)\,\textit{\textbf v}_\epsilon
=(F_i^\epsilon(\textit{\textbf z}_\epsilon^k)-F_i^\epsilon(\textit{\textbf z}_\epsilon))
\,\textit{\textbf v}_\epsilon^k+F_i^\epsilon(\textit{\textbf z}_\epsilon) \,
(\textit{\textbf v}_\epsilon^k-\textit{\textbf v}_\epsilon),
\end{eqnarray*}
by Proposition \ref{Proposition 5.3}, the regularity of $F_i^\epsilon(\cdot,\cdot)$, and the Schwarz's inequality,
\begin{eqnarray}\label{eq97}
&&\int_0^T(F_i^\epsilon(\textit{\textbf z}_\epsilon^k)\,\textit{\textbf v}_\epsilon^k
-F_i^\epsilon (\textit{\textbf z}_\epsilon)\,\textit{\textbf v}_\epsilon,
\textit{\textbf w}^j)\,\phi(t)\, dt \nonumber \\
&&\hspace{.5cm}\leq C\int_0^T\|F_i^\epsilon(\textit{\textbf z}_\epsilon^k)- F_i^\epsilon
(\textit{\textbf z}_\epsilon)\|_\infty \|\textit{\textbf v}_\epsilon^k\,(t)\|\, dt\nonumber\\
&&\hspace{1cm} +C\int_0^T \|F_i^\epsilon(\textit{\textbf z}_\epsilon)\|_\infty
\|\textit{\textbf v}_\epsilon^k\,(t)-\textit{\textbf v}_\epsilon\,(t)\|\,dt, \nonumber \\
&&\hspace{.5cm}\leq  C\int_0^T\|F_i^\epsilon(\textit{\textbf z}_\epsilon^k)- F_i^\epsilon
(\textit{\textbf z}_\epsilon) \|_\infty dt\nonumber\\
&&\hspace{1cm} +C(T,\epsilon)\,(\int_0^T
\|\textit{\textbf v}_\epsilon^k\,(t)-\textit{\textbf v}_\epsilon\,(t)\|^2\,dt)^{1/2}.
\end{eqnarray}
Therefore, from (\ref{eq86})-(\ref{eq91}), (\ref{eq93})-(\ref{eq97}) and Lemma \ref{Lemma 5.7}, we conclude 
that the system (\ref{eq51})-(\ref{eq54}) converges to the system (\ref{eq46})-(\ref{eq49}).
Thus, the proof of the Theorem \ref{Theorem 5.1} is complete. 

\section{Reproductive weak solution of the solidification problem}
In this section we will prove our main result.

\begin{theorem}\label{Theorem 6.1} 
Under the hypotheses ({\bf S$_1$}), ({\bf S$_2$}) and ({\bf S$_3$}), the problem (\ref{eq8})-(\ref{eq10c})
admits at least one weak reproductive solution.
\end{theorem}

\begin{proof} From Theorem \ref{Theorem 5.1} and Proposition \ref{Proposition 5.3}, for $\epsilon \in (0,1]$, 
any reproductive weak solution $(c_\epsilon, \theta_\epsilon, \textit{\textbf v}_\epsilon)$ of the problem 
(\ref{eq30})-(\ref{eq32c}) is uniformly bounded with respect to $\epsilon$. Then, similarly to Lemma \ref{Lemma 5.5} 
and Lemma \ref{Lemma 5.7}, we can conclude that there exists $(c ,\theta,\textit{\textbf v})$ and a sequence, 
still indexed by $\epsilon$, such that
\begin{eqnarray}
&&c_\epsilon \ \rightarrow c,  \quad \theta_\epsilon \ \rightarrow \theta \quad \mbox{ weakly in } 
L^2(H^1), \ \mbox{ strongly in } L^2(L^2), \label{seq1}\\
&&\textit{\textbf v}_\epsilon \ \rightarrow \textit{\textbf v} \quad  \mbox{ weakly in } L^2(\textit{\textbf H}_v), 
\mbox{ strongly in } L^2(\textit{\textbf H} ). \label{seq2}
\end{eqnarray}
We check now that $(c, \theta, \textit{\textbf v})$ is a weak reproductive solution to problem (\ref{eq8})-(\ref{eq10c}).

Let $\textit{\textbf{w}} \in \textit{\textbf C}^\infty_0(\Omega)$ with compact support 
$K \subset \Omega_{ml}$. Then, for an certain $\delta > 0$ and for every $x \in K$,
\begin{equation}\label{eq98}
 f_s(c(x,t),\theta(x,t)) < 1 - \delta \quad \mbox{ a.e. in } \ (0,T),
\end{equation}
and this implies that
\begin{equation}\label{eq99}
\|F_i(c(t),\theta(t))\|_{L^\infty(K)} < C(\delta) \quad \mbox{ a.e. in } \ (0,T).
\end{equation}

Due to the uniform convergence of $f_s(c_\epsilon,\theta_\epsilon)$ toward $f_s(c, \theta)$ for any 
compact subset of $\Omega_{ml}$, given $\frac{\delta}{2}$ there exist an $\epsilon_\delta$ such
that $\forall \,\epsilon \in (0,\epsilon_\delta)$ and $\forall \, x \in K$,
\begin{equation}\label{eq100}
|f_s(c_\epsilon(x,t),\theta_\epsilon(x,t))-f_s(c(x,t),\theta(x,t)| < \frac{\delta}{2},
\end{equation}
thus, from (\ref{eq100}) and (\ref{eq98}), $ \forall \, x \in K$ and a.e. in $(0,T)$ we have
\begin{equation}\label{eq101}
f_s(c_\epsilon(x,t),\theta_\epsilon(x,t)) < \frac{\delta}{2}+ f_s(c(x,t), \theta(x,t) < 1- \frac{\delta}{2}.
\end{equation}
Consequently, considering (\ref{eq5}), (\ref{eq29}), (\ref{eq98}) and (\ref{eq101}), when $\epsilon \rightarrow 0$,
\begin{equation}\label{eq102}
F_i^\epsilon(c_\epsilon(x,t),\theta_\epsilon(x,t)) \longrightarrow
F_i (c(x,t),\theta(x,t)) \ \mbox{in } C^0(K) \ \mbox{ a.e. in } \ (0,T).
\end{equation}
Then, taking into account (\ref{eq99}) and similarly as (\ref{eq97}), we have
\begin{eqnarray}\label{eq103}
&&\int_0^T(F_i^\epsilon(c_\epsilon,\theta_\epsilon)\,
\textit{\textbf v}_\epsilon\,(t)- F_i (c, \theta)\,\textit{\textbf v}
\,(t), \,\textit{\textbf w}^i)\,\phi(t)\, dt \nonumber \\
&&\hspace{.5cm}\leq C \int_0^T \|F_i^\epsilon(c_\epsilon,\theta_\epsilon)
- F_i (c, \theta)\|_{L^\infty(K)}dt\nonumber\\
&&\hspace{1cm}+C(\delta)(\int_0^T\|\textit{\textbf v}_\epsilon\,(t) 
-\textit{\textbf v}\,(t)\|^2 \,dt)^{1/2}.
\end{eqnarray}
Therefore, using (\ref{eq102}), (\ref{eq103}), (\ref{seq1}), (\ref{seq2}), and passing to the limit 
in the equations (\ref{eq33})-(\ref{eq35}), we obtain (\ref{eq15})-(\ref{eq17}), where by density 
the equation (\ref{eq17}) holds for any $\textit{\textbf w} \in \textit{\textbf H}_v$ with
supp $\textit{\textbf w} \subset \Omega_{ml}$.

To check that $\textit{\textbf{v}}=0$ in $\Omega_s \times (0,T)$ take a compact set $K \subset \Omega_s$. 
Since the solid domain is open, there exists an $\epsilon_K > 0$ such that
$\,f_s(c_\epsilon(x,t),\theta_\epsilon(x,t)) = 1$ if $x \in K$  a.e. in $(0,T)$, whenever  
$\epsilon \in (0,\epsilon_K)$, and this implies that
\begin{equation}\label{eq104}
F_i^\epsilon(c_\epsilon(x,t),\theta_\epsilon(x,t))= \frac{C_0}
{\epsilon^3}\quad \mbox{in } \, x \in K, \ \mbox{ a.e. in } \ (0,T).
\end{equation}

Choosing $\textit{\textbf w} = \textit{\textbf v}_\epsilon$ in (\ref{eq35}),  observing (\ref{eq66}) 
and (\ref{eq104}), we have
\begin{eqnarray}\label{eq105}
&&\frac{1}{2}\frac{d}{dt} \|\textit{\textbf v}_\epsilon(t)\|^2
+\frac{3\,\nu}{4} \|\nabla \textit{\textbf v}_\epsilon(t)\|^2
+\frac{C_0}{\epsilon^3}\,\|\textit{\textbf v}_\epsilon(t)\|^2_{L^2(K)}\leq C. 
\end{eqnarray}
By integrating (\ref{eq105}) respect to $t$ from $0$ to $T$, and since
$\textit{\textbf v}_\epsilon(0)=\textit{\textbf v}_\epsilon(T)$, we get
\begin{equation}\label{eq106}
\frac{C_0}{\epsilon^3}\, \int_0^T \|\textit{\textbf v}_\epsilon(t)\|^2_{L^2(K)}\,dt \leq C\, T,
\end{equation}
where $C$ is independent of $\,\epsilon$. Thus, as $ \,\epsilon$ vanishes $\displaystyle \,
\frac{C_0}{\epsilon^3} \rightarrow \infty$, then (\ref{eq106}) implies that
$\displaystyle \, \int_0^T \|\textit{\textbf v}_\epsilon(t)\|^2_{L^2(K)}
\,dt$ to converge to $0$ and consequently $\|\textit{\textbf v}_\epsilon(t)\|_{L^2(K)}$ converge 
to $0$. Therefore, $\textit{\textbf v}=0$ in $K$ a.e. in $(0,T)$, and the arbitrary choice of 
$K$ implies $\textit{\textbf v}= 0$ in $\Omega_s \times (0,T)$. Thus, we complete the proof 
of the Theorem \ref{Theorem 6.1}. 
\end{proof}

%


\begin{thebibliography}{99}
\bibitem{Ah} Ahmad, N.:  Numerical simulation of transport processes in
multicomponent systems related to solidification problems,
Ph.D.Thesis DMA-EPFL N$^{\underline\circ}$ 1349 (1995).

\bibitem{Ba-Du-Or} Badillo, A., Dur\'an, M., Ortega-Torres, E.:
C\'alculo de inestabilidades de un proceso de solidificaci\'on en
dominios de simetr\'{\i}a cil\'{\i}ndrica, Parte II: Resultados Numericos.
Revista Internacional de M\'etodos Num\'ericos para C\'alculo y Dise\~no
en Ingenier\'{\i}a, Vol. 21, 4, 307-325 (2005).

\bibitem{Bl-Gas-Ra} Blanc, Ph., Gasser, L., Rappaz, J.: Existence for
a Stationary Model of Binary Alloy Solidification. Mathematical
Modelling and Numerical Analysis,  Vol. 29, 6, 687-699 (1995).

\bibitem{Co-Le} Combeau, H., Lesoult, G.:  Simulation of freckles
formation and related segregation during directional solidification
of metallic alloys. Modelling of casting, welding and advanced
solidification processes 6, Voller V., Piwonka T.S.
\& Katgerman L. eds., 201-208 (1993).

\bibitem{Du-Or-Ra} Dur\'an, M., Ortega-Torres, E., Rappaz, J.:
Weak Solution of a Stationary Convection-Diffusion Model Describing
Binary Alloy Solidification Processes. Mathematical and Computer
Modelling, Vol. 42, 11/12, 1269-1286 (2005).

\bibitem{Du-Or-St} Dur\'an, M., Ortega-Torres, E., Stein, R.:
C\'alculo de inestabilidades de un proceso de solidificaci\'on en
dominios de simetr\'{\i}a cil\'{\i}ndrica, Parte I: Formulaci\'on
del modelo. Revista Internacional de M\'etodos Num\'ericos para
C\'alculo y Dise\~no en Ingenier\'{\i}a, Vol. 17, 2,
127-148 (2001).

\bibitem{Er} Eringen, A.: Theory of thermo-microstretch fluids and bubbly liquids, Int. J. Engng. Sci. 28, 133-143, (1990).

\bibitem{Ga} Gaillard, F.: Mod\'elisation et analyse num\'erique
d'instabilit\'es d\'evelop\'ees lors de la solidification
d'alliages binaires, Ph.D. Tesis DMA-EPFL, N$^{\underline\circ}$
1868 (1998).

\bibitem{Ga-Ra} Gaillard, F.,  Rappaz, J.: Analysis and
numerical simulation for models of binary alloy solidification,
Periaux Contributed Book, John Wiley \& Sons Ltd (1997).

\bibitem{Gas} Gasser, L.: Existence analysis and numerical
schemes for models of binary alloy solidification, Ph.D.Thesis
DMA-EPFL N$^{\underline\circ}$ 1421 (1995).

\bibitem{Hi-Lo-Ro} Hills, R.N., Loper, D.E., Roberts, P.H.:  A thermodynamically consistent model of a mushy zone. Quarterly
J. Mech. Appl. Math., 36, 505-539 (1983).

\bibitem{La} Ladyszhenskaya, O.A.: The mathematical theory of viscous
incompressible flow, Second edition, Gordon and Breach, New York (1969).

\bibitem{Mil} Milne-Thomson, L.M.: Theoretical hydrodinamics. Dover Publications, Inc., New York 1996.

\bibitem{Ra-Vo} Rappaz, M., Voller, V.R.:  Modelling of
micro-macrosegregation in solidification processes. Metall. Trans.
A-Physical Metallurgy and Materials, Vol. 21, 3, 749-753 (1990).

\bibitem{Sha}  Shapiro, A.H.: The dynamics and thermodynamics of compressible flow. Ronald Press, New
York 1953.

\bibitem{Vo-Pr} Voller, V.R., Prakash, C.: A fixed
grid numerical modelling methodology for convection-diffusion
mushy region phase-change problems. Int. J. Heat Mass Transfer,
30, 1709-1719 (1987).
\end{thebibliography}
\end{document}